\documentclass[reqno]{amsart}
\usepackage{amsmath, amssymb, amsthm, epsfig}
\usepackage{hyperref, latexsym}
\usepackage{url}
\usepackage[mathscr]{euscript}
\usepackage{color}
\usepackage{fullpage} 
\usepackage{setspace}

\def\today{\ifcase\month\or
  January\or February\or March\or April\or May\or June\or
  July\or August\or September\or October\or November\or December\fi
  \space\number\day, \number\year}
\DeclareMathOperator{\sgn}{\mathrm{sgn}}

 \newtheorem{theorem}{Theorem}
  
 \newtheorem{lemma}[theorem]{Lemma}
 \newtheorem{proposition}[theorem]{Proposition}
 \newtheorem{corollary}[theorem]{Corollary}
 \theoremstyle{definition}

 \theoremstyle{remark}

\newcommand{\Tau}{\mathcal{T}}
\newcommand{\PW}{\text{\rm PW}}
 \newcommand{\ft}{\widehat}
 \newcommand{\mc}{\mathcal}
 
 \newcommand{\B}{\mc{B}}

 \newcommand{\LL}{\mc{L}}

 \newcommand{\U}{\mc{U}}

 \newcommand{\C}{\mathbb{C}}
 \newcommand{\R}{\mathbb{R}}
 \newcommand{\N}{\mathbb{N}}
 
 \newcommand{\Z}{\mathbb{Z}}

 \newcommand{\p}{\varphi}

 \newcommand{\ds}{\text{\rm d}s}

 \newcommand{\dt}{\text{\rm d}t}

 \newcommand{\dx}{\text{\rm d}x}

    \renewcommand{\d}{\text{\rm d}}

\newcommand{\Rep}{\textrm{Re}}

\newcommand{\ov}{\overline}
\renewcommand{\H}{\mc{H}}
\newcommand{\im}{{\rm Im}\,}

\newcommand{\hf}{\tfrac{1}{2}}

\newcommand{\<}{<\hspace{-1.5 mm}<}
\begin{document}

\title[]{Interpolation Formulas with Derivatives in de Branges Spaces}
\author[Gon\c{c}alves]{Felipe Gon\c{c}alves}
\date{\today}
\address{IMPA - Instituto de Matem\'{a}tica Pura e Aplicada - Estrada Dona Castorina, 110, Rio de Janeiro, RJ, Brazil 22460-320}
\email{ffgoncalves@impa.br}

\allowdisplaybreaks
\numberwithin{equation}{section}

\begin{abstract}
The purpose of this paper is to prove an interpolation formula involving derivatives for entire functions of exponential type. 
We extend the interpolation formula derived  by J. Vaaler in \cite[Theorem 9]{V} to general $L^p$ de Branges spaces.
We extensively use techniques from de Branges' theory of Hilbert spaces of entire functions as developed in \cite{B2}, but a crucial passage involves the Hilbert--type inequalities as derived in \cite{CLV}.
We give applications to homogeneous spaces of entire functions that involve Bessel functions and we prove a uniqueness result for extremal one-sided band-limited approximations of radial functions in Euclidean spaces.
\end{abstract}

\maketitle

\section{Introduction}

\subsection{Background}
An entire function $F: \C \to \C$, not identically zero, is said to be of {\it exponential type} if
\begin{equation*}
\tau(F) = \limsup_{|z| \to \infty} |z|^{-1}\, \log |F(z)| < \infty.
\end{equation*}
In this case, the non-negative number $\tau(F)$ is called the exponential type of $F$.\smallskip

In \cite[Theorem 9]{V}, J. Vaaler proved that if $F(z)$ is an entire function of exponential type at most $2\pi$ that belongs to $L^p(\R,\dx)$ for some $p\in(0,\infty)$ then
\begin{equation}\label{vaaler-int-form}
{F(z)}=\frac{\sin^2(\pi z)}{\pi^2}\sum_{n\in\Z}\bigg\{\frac{F(n)}{(z-n)^2} + \frac{F'(n)}{(z-n)}\bigg\},
\end{equation}
\noindent where the sum converges uniformly on compact sets of $\C$.
Furthermore, in the case $p=2$, it can be proven using Paley-Wiener spaces techniques that the formula also converges in the $L^2(\R,\dx)$-norm. Also, a similar formula holds if we substitute the integers by any translation of them.

\smallskip

Given a number $\tau>0$ and $p\in(0,\infty]$ the classical Paley-Wiener space $\PW(\tau,p)$ is defined as the space of entire functions $F(z)$ of exponential type at most $\tau$ that belong to $L^p(\R,\dx)$. In the case $p=2$ this is a Hilbert space with the standard $L^2(\R,\dx)$-inner product and it can be proven that convergence in the space implies uniform convergence on compact sets of $\C$. Based on the Hilbert space setting, the natural environment to extend the interpolation formula (\ref{vaaler-int-form}) would be the {\it de Branges spaces} of entire functions as developed by L. de Branges in \cite{B2}, since they generalize the Paley-Wiener spaces.\smallskip

Intuitively, a de Branges space can be seen as a weighted Paley-Wiener space. Given a Hermite-Biehler function $E(z)$ (see the definition in \S\ref{deBranges}) and a number $p\in(0,\infty]$, the space $\H^p(E)$ is a space of entire functions $F(z)$ that satisfies a certain growth condition relatively to $E(z)$ and such that $F/E$ belongs to $L^p(\R,\dx)$.\smallskip

Formula (\ref{vaaler-int-form}) is useful in applications to approximation theory. In \cite{GV}, S. Graham and J. Vaaler used this formula to construct extremal one-sided approximations of exponential type to a given real-valued function $g(x)$. Under certain restrictions on $g(x)$, they characterized the pair of entire functions $M(z)$ and $L(z)$ of exponential type at most $2\pi$ that satisfies $L(x)\leq g(x) \leq M(x)$ for all real $x$ minimizing the quantities 
\begin{equation*}
\int_\R \{M(x)-g(x)\}\,\dx \,\,\,\, \mbox{ and } \,\,\,\, \int_\R\{g(x)-L(x)\}\,\dx.
\end{equation*}
\indent In  \cite{CLV}, E. Carneiro, F. Littmann and J. Vaaler applied the same methods to produce extremal one-sided band-limited approximations for functions $g(x)$ that are in some sense subordinated to the Gaussian function. Later in \cite{GMK}, F. Gon\c{c}alves, M. Kelly and J. Madrid extended their results to the several variables regime. Other important works that apply such interpolation formulas are \cite{CL, CV2, V}.

\smallskip

If, instead of the $L^1(\R,\dx)$-norm, one decides to minimize a weighted norm $L^1(\R,\d\mu(x))$, where $\mu(x)$ is a non-decreasing function on the real line, the Fourier transform tools are no longer available. The alternative theory to approach these new extremal problems is the theory of de Branges spaces. Several works have been done in this direction, see \cite{CCLM,CG,CL3,CL4,HV,L2,LSp}. The methods used in these later works were very different than the previous ones, since generalizations of the formula (\ref{vaaler-int-form}) to de Branges spaces were not known at the time. These special functions $M(z)$ and $L(z)$ have been used in a variety of interesting applications in number theory and analysis, for instance in connection to: large sieve inequalities \cite{HV, V}, Erd\"{o}s-Tur\'{a}n inequalities \cite{CV2,V}, Hilbert-type inequalities \cite{CL3, CLV, CV2, GV, L3, V}, Tauberian theorems \cite{GV} and bounds in the theory of the Riemann zeta-function and general $L$-functions \cite{CC,CCLM,CCM,CCM2, CS,Ga,GG}.

\smallskip

\subsection{De Branges Spaces}\label{deBranges}

In order to properly state our results we need to briefly review the main concepts and terminology of the theory of $L^p$ de Branges spaces (see \cite{Bar2,B2}).  \smallskip

Throughout the text we denote by 
\begin{equation*}
\U = \{z \in \C; \, \im(z) >0\}
\end{equation*}
the open upper half-plane. An analytic function $F: \U \to \C$ has {\it bounded type} if it can be written as a quotient of two functions that are analytic and bounded in $\U$ (or equivalently, if $\log |F(z)|$ admits a positive harmonic majorant in $\U$). If $F: \U \to \C$ is not identically zero and has bounded type, from its Nevanlinna factorization \cite[Theorems 9 and 10]{B2}, the number 
\begin{equation*}
v(F) = \limsup_{y \to \infty} \,y^{-1} \,\log |F(iy)|,
\end{equation*}
called the {\it mean type} of $F$, is finite. 
It can be proven that the set of functions of bounded type in $\U$ is an algebra and 
\begin{equation}\label{mean-type-relations}
v(FG)=v(F)+v(G) \,\,\,\, \mbox{ and } \,\,\,\,\, v(F+G)\leq \max\{v(F),v(G)\},
\end{equation}
if $F(z)$ and $G(z)$ are of bounded type  in $\U$ (see \cite[Problem 29]{B2}).

\smallskip

If $E: \C \to \C$ is entire, we define the entire function $E^*: \C \to \C$ by $E^*(z) = \ov{E(\ov{z})}$. A {\it Hermite-Biehler} function $E: \C \to \C$ is an entire function that satisfies the basic inequality
\begin{equation*}
|E^*(z)| < |E(z)|
\end{equation*}
for all $z \in \U$. 
Associated to $E(z)$, we define the companion functions
\begin{equation*}\label{companion-func}
A(z) := \frac12 \big\{E(z) + E^*(z)\big\} \ \ \ {\rm and}  \ \ \ B(z) := \frac{i}{2}\big\{E(z) - E^*(z)\big\}.
\end{equation*}
\noindent Note that $A(z)$ and $B(z)$ are real entire functions with only real zeros and  $E(z) = A(z) - iB(z)$.
Similarly, if $\alpha$ is a real number, we write 
\begin{equation}\label{companion-func-alpha}
e^{i\alpha}E(z)=A_\alpha(z)-iB_\alpha(z)
\end{equation}
where $A_\alpha(z)$ and $B_\alpha(z)$ are real entire functions. Note that $B_{\alpha-\pi/2}(z)=A_\alpha(z)$.\smallskip

We denote by $\p(z)$ the phase function associated to $E(z)$. This function is defined by the
condition $e^{i\p(x)}E(x)\in\R$ for all real $x$. It can be shown that $\p(z)$ is analytic on a neighborhood of $\R$, any two of such functions differ by an integer multiple of $\pi$, and $\p'(t)>0$ for all real $t$ (see \cite[Problem 48]{B2} and \cite{HV}). For a given real number $\alpha$ we define
\begin{equation*}\label{tau-alpha}
\Tau(\alpha)=\{x\in\R:\p(x)\equiv \alpha \, (\!\!\! \!\!\! \mod \pi)\}
\end{equation*}
and we note that $\Tau(\alpha)$ is the set of all real zeros of $B_\alpha(z)/E(z)$.\smallskip

If $E(z)$ is a Hermite-Biehler function and $p\in(0,\infty]$, we define the $L^p$ {\it de Branges space} $\H^p(E)$ as the space of entire functions $F:\C \to \C$ such that $F/E$ and $F^*/E$ have bounded type in $\U$ with non-positive mean type and 
\begin{equation*}\label{Intro_Def_norm_E}
\|F\|_{E,p}=\bigg(\int_\R|F(x)/E(x)|^p\dx\bigg)^{1/p} <\infty 
\end{equation*}
if $p$ is finite, and 
\begin{equation*}\label{Intro_Def_norm_E_infty}
\|F\|_{E,\infty}=\sup_{x\in\R} |F(x)/E(x)| <\infty 
\end{equation*}
if $p=\infty$. When $p\geq 1$ these are Banach spaces (see Section \ref{Lp-form}) and when $p=2$ (we write $\H(E)=\H^2(E)$ and $\|\cdot\|_{E,2}=\|\cdot\|_E$) this forms a Hilbert space with inner product given by
\begin{equation*}
\langle F, G \rangle_E = \int_{-\infty}^{\infty} F(x)\, \ov{G(x)}\, |E(x)|^{-2}\,\dx.
\end{equation*}

The remarkable property about these spaces is that, for each $w \in \C$, the evaluation map $F \mapsto F(w)$ is a continuous linear functional. It can be shown, using Cauchy's formula for the upper half-plane (see \cite[Theorems 12 and 19]{B2}), that the function
\begin{align}\label{reprod-kern}
K(w,z)  = \frac{E(z)E^*(\ov{w}) - E^*(z)E(\ov{w})}{2\pi i (\ov{w}-z)} = \frac{B(z)A(\ov{w}) - A(z)B(\ov{w})}{\pi (z - \ov{w})}
\end{align}
is a {\it reproducing kernel} for these spaces. That is, for any $w\in\C$ and any $p\in[1,\infty)$ the function $K(w,\cdot)$ belongs to $\H^{p'}(E)$, where $1/p+1/p'=1$, and 
\begin{equation}\label{repord-kern-id}
F(w)=\langle F, K(w,\cdot) \rangle_E = \int_{-\infty}^{\infty} F(x)\, \ov{K(w,x)}\, |E(x)|^{-2}\,\dx,
\end{equation}
for each $F \in \H^p(E)$. Note that, by Cauchy-Schwarz inequality, we obtain
\begin{equation}\label{swartz-ineq}
|F(w)|\leq \|F\|_{E,p}\|K(w,\cdot)\|_{E,p'}.
\end{equation}
It can be shown that $w\mapsto \|K(w,\cdot)\|_{E,p'}$ is continuous, hence we see that convergence in the space implies uniform convergence on compact sets of $\C$.\smallskip

From the reproducing kernel property we have
\begin{equation*}
K(w,w) = \langle K(w, \cdot),  K(w, \cdot) \rangle_E =\|K(w,\cdot)\|_E\geq 0\,,
\end{equation*}
and one can easily show that $K(w,w) =0$ if and only if $w \in \R$ and $E(w) =0$ (see for instance \cite[Lemma 11]{HV} or \cite[Problem 45]{B2}).


For a given $a>0$, we define the Paley-Wiener space $\PW(a,p)=\H^p(e^{-iaz})$. By Krein's theorem (see \cite{Kr} and \cite[Lemma 12]{HV}) this space coincides with the space of entire functions $F(z)$ of exponential type at most $a$ such that $F\in L^p(\R,\dx)$.
\smallskip

In Section \ref{Lp-form} we give a different approach for defining the spaces $\H^p(E)$ connecting with the theory of Hardy spaces in the upper half-plane. Also in Section \ref{Lp-form} we comment about the proof of completeness of these spaces.\smallskip

\subsection{Main Results}

We say that a de Branges space $\H^p(E)$ is closed by differentiation if  $F'\in\H^p(E)$ whenever $F\in\H^p(E)$. By (\ref{swartz-ineq}) we conclude that for $p\in[1,\infty)$ convergence in the space implies uniform convergence on compacts sets of $\C$, hence the differentiation operator is always a closed operator. Thus, by the Closed Graph Theorem, it is continuous whenever it is everywhere defined. 

Recall that we omit the superscript $p$ in $\H^p(E)$ only when $p=2$, that is, we write $\H(E)=\H^2(E)$. The crucial idea for the main result of the paper is to proof an interpolation formula with derivatives for functions in the space $\H(E^2)$, not in $\H(E)$. As in the Vaaler's proof, the natural space for the correct interpolation formula was $PW(2\pi,2)=\H([e^{-i\pi z}]^2)$. Also note that $E(z)^2=A(z)^2-B(z)^2-2iA(z)B(z)$, thus the condition $AB\notin \H(E^2)$ will be necessary for the main result (see formula \eqref{norm-basis}).

The following theorem is the main result of the paper.

\begin{theorem}\label{gen-int-form-der}
 Let $E(z)$ be a Hermite-Biehler function such that $\H(E^2)$ is a de Branges space closed by differentiation. Suppose that for a real number $\alpha$ we have $A_\alpha B_\alpha \notin\H(E^2)$ and  $\p'(x)$ is bounded away from zero over $\Tau(\alpha)$.
Then, if $p\in[1,2]$ and $F\in\H^p(E^2)$, we have
 \begin{equation}\label{gen-int-form}
  F(z)=B_\alpha(z)^2\sum_{t\in\Tau(\alpha)}\bigg\{\frac{F(t)}{B_\alpha'(t)^2(z-t)^2}+ 
\frac{F'(t)B_\alpha'(t)-F(t)B_\alpha''(t)}{B_\alpha'(t)^3(z-t)}  \bigg\}\,,
 \end{equation}
\noindent where the sum converges uniformly on compact sets of $\C$.
This formula is also valid for $p\in(2,\infty)$ if we additionally assume that $v(E^*/E)<0$.
\end{theorem}

\noindent {\it Remark:} We note that there exists at most one $\alpha$ modulo $\pi/2$ such that $A_\alpha B_\alpha \in\H(E^2)$ otherwise $E^2(z)$ would belong to $\H(E^2)$, which is an absurd. In the paper \cite{Bar}, A. Baranov proved that if $E'/E$ belongs to the Hardy space $H^\infty(\U)$ (see Section \ref{Lp-form}) then the differentiation operator is continuous in  $\H(E)$. He also concluded that this condition is necessary if we assume $v(E^*/E)<0$ (see also \cite{Bar2}).

\smallskip

We highlight the fact that Vaaler's proof of \eqref{vaaler-int-form} in \cite{V} relies heavily on Fourier analysis, a tool that is not available in this general setting. Thus, our main challenge here (and motivation to consider this problem) is two-fold: (i) to find a Fourier analysis-free proof of \eqref{vaaler-int-form}; (ii) to extend this proof to the general setting. This is carried out in Sections \ref{L2-form} and \ref{Lp-form}.

\smallskip

We present here a corollary of this result related to sampling theory.

\begin{corollary}\label{cor}
Let $E(z)=A(z)-iB(z)$ be an Hermite-Biehler function such that $\PW(a,2)=\H(E^2)$ as sets. Suppose that for some constant $M>0$, $|A(t)|\leq M$ whenever $B(t)=0$. Then there exists a constant $C>0$ such that
\begin{equation}\label{sampling}
C^{-1}\int_\R|F(t)|^2\dt \leq  \sum_{B(t)=0} \{|F(t)|^2+|F'(t)|^2\} \leq C\int_\R|F(t)|^2\dt
\end{equation}
for every $F\in\PW(a,2)$. Furthermore, if $\{t_n\}_{n\in\N}$ is an enumeration of the real zeros of $B(z)$ then for every pair $(p_n)\in l^2(\N)$ and $(q_n)\in l^2(\N)$ of complex sequences there exists an unique function $F\in\PW(a,2)$ such that $F(t_n)=p_n$ and $F'(t_n)=q_n$ for all $n$.
\end{corollary}

\noindent {\it Remark:} Following the ideas of J. Ortega-Cerd\`a and  K. Seip in \cite{OS}, Corollary \ref{cor} gives a sufficient condition for a sequence of points to be {\it sampling with derivates} for $\PW(a,2)$. We say that a discrete set of real points $\Lambda$ is {\it sampling with derivatives} for $\PW(a,2)$ if there exists a constant $C>0$ such that

\begin{equation*}
C^{-1}\int_\R|F(t)|^2\dt \leq  \sum_{t\in\Lambda} \{|F(t)|^2+|F'(t)|^2\} \leq C\int_\R|F(t)|^2\dt
\end{equation*}
for every $F\in\PW(a,2)$. Also, in the paper \cite{LS}, Y. Lyubarskii and K. Seip give necessary and sufficient conditions for a Hermite-Biehler function $E(z)$ to satisfy $\PW(a,2)=\H(E^2)$.

\subsection{Organization of the Paper}
In Section \ref{L2-form} we prove Theorem \ref{gen-int-form-der} for the case $p=2$ using de Branges space techniques. In Section \ref{Lp-form} we review the aspects of $L^p$ de Branges spaces and provide the full proof of Theorem \ref{gen-int-form-der}. In Part 1 of Section \ref{app} we give a quick review of homogeneous spaces and derive interpolation formulas for these spaces, which fully generalize the interpolation results derived in \cite{V}. Finally, in Part 2 of Section \ref{app} we provide a direct application of our formulas, proving a uniqueness result concerning best one-sided approximations by band-limited functions in Euclidean spaces.

\subsection{Notation Remark}
Given two positive quantities $Q$ and $Q'$ and $N$ real quantities $r_1,...,r_N$ we write 
$Q\<_{r_1,...,r_N}Q'$ when $Q\leq C(r_1,...,r_N)Q'$ where $C:\Omega\subset\R^N\to(0,\infty)$ is some positive function. We also write $Q\simeq_{r_1,...,r_N} Q'$ when both $Q\<_{r_1,...,r_N}Q'$ and $Q'\<_{r_1,...,r_N}Q$ hold. Often, the quantities $Q$ and $Q'$ will depend on a function $F$, that is 
$Q=Q(F)$ and $Q'=Q'(F)$. We write $Q(F)\<Q'(F)$ when there exists a constant $C>0$, which does not depend on $F$, such that $Q(F)\leq CQ'(F)$.

\medskip

\section{Interpolation Formulas in de Branges Spaces}\label{L2-form}

Without the Fourier transform theory  we need to use a different approach than that used by J. Vaaler in \cite{V}. The recipe to extend formula (\ref{vaaler-int-form}) is

\begin{enumerate}
 \item Substitute the function $\sin(\pi z)$ by the companion function $B_\alpha(z)$ defined in (\ref{companion-func-alpha}) associated with a Hermite-Biehler function $E(z)$.
 \item Prove that formula (\ref{gen-int-form}) is valid for a dense set of functions in $\H(E^2)$.
 \item Deduce inequalities that guarantee that the formula will remain valid when we pass to the limit. 
\end{enumerate}\smallskip

 \noindent In the last step of this recipe we shall use the Hilbert-type inequalities as derived in \cite{CLV}.

\subsection{Preliminary Results}
 Let $E(z)$ be an Hermite-Biehler function and recall that we write $\p(z)$ for the phase function. If $t$ and $\alpha$ are real numbers such that $\varphi(t)\equiv \alpha \, \,(\!\!\! \mod \pi)$ we have
 \begin{equation}\label{phi-ident}
 \p'(t)=\pi K(t,t)/|E(t)|^2=\frac{B'_\alpha(t)}{A_\alpha(t)}>0,
 \end{equation}
 where $K(w,z)$ is defined in (\ref{reprod-kern}) (see \cite[Problem 48]{B2}).
We also have 
\begin{equation*}\label{Re-A/B}
 0<\frac{|E(z)|^2-|E^*(z)|^2}{2y|B(z)|^2}=\Rep \,\, i\frac{A(z)}{B(z)}
\end{equation*} 
if $y>0$. In a similar way $\Rep [-iB(z)/A(z)]>0$ if $y>0$. 

\smallskip

Throughout the rest of this paper we will always 
denote by $\{t_n\}$ the points such that $\varphi(t_n)=\pi n$ for all $n\in\Z$ 
and $\{s_n\}$ the points such that $\varphi(s_n)=\pi/2+n\pi$ for all $n\in\Z$. These points are respectively all the real zeros of $B(z)/E(z)$ and $A(z)/E(z)$. Also these zeros are simple. 

To see this, suppose that $t_n$ is a zero of $E(z)$ of order $m\geq 0$ and of $B(z)$ of order $m+l\geq 1$. We claim that $l=1$. If $m=0$, then by \eqref{phi-ident} and \eqref{companion-func-alpha} we trivially have $l=1$. If not, then $\tilde{E}(z)=E(z)/(z-t_n)^m$ is a Hermite-Biehler function and $\tilde{E}(t_n)\neq 0$, hence by the previous argument $t_n$ is a simple zero of $B(z)/(z-t_n)^m$ and thus $l=1$. We conclude that the points $\{t_n\}$ and $\{s_n\}$ are respectively simple zeros and simple poles of $B(z)/A(z)$.
\smallskip

According to \cite[Theorem 22]{B2}, for every real number $\alpha$ the set of functions 

$$
\bigg\{\frac{B_\alpha(z)}{(z-t)}\bigg\}_{t\in\Tau(\alpha)}
$$
is an orthogonal set in $\H(E)$ and
\begin{equation}\label{norm-basis}
\|F\|_E^2 \geq \sum_{t\in\Tau(\alpha)} \frac{|F(t)|^2}{K(t,t)} = \pi\sum_{t\in\Tau(\alpha)}  \frac{|F(t)|^2}{B'_\alpha(t)A_\alpha(t)},
\end{equation}\smallskip
where equality holds if and only if $B_\alpha\notin \H(E)$. We have the following lemma.

\begin{lemma}\label{int-form-B/A}
Let $E(z)$ be a Hermite-Biehler function with no real zeros. If $A\notin\H(E)$ then

\begin{enumerate}
\item For all complex numbers $z$ and $w$ not equal to any $s_n$ we have 
\begin{equation}\label{B/A-form-0}
\frac{B(z)/A(z)-B(\overline{w})/A(\overline{w})}{\overline{w}-z}=\sum_n\frac{B(s_n)}{A'(s_n)(z-s_n)(\overline{w}-s_n)}.
\end{equation}
\item For all $s_j$ we have
\begin{equation}\label{B/A-form-1}
\frac{B(z)}{A(z)} = \frac{B'(s_j)}{A'(s_j)}- \frac{B(s_j)A''(s_j)}{2A'(s_j)^2} + \frac{B(s_j)}{A'(s_j)(z-s_j)} + \sum_{n\neq j} \frac{B(s_n)}{A'(s_n)}\bigg( \frac{1}{z-s_n} + \frac{1}{s_n-s_j}  \bigg).
\end{equation}
\item For all $t_j$ we have
\begin{equation}\label{B/A-form-2}
\frac{B(z)}{A(z)}=\sum_{n} \frac{B(s_n)}{A'(s_n)}\bigg( \frac{1}{z-s_n} + \frac{1}{s_n-t_j}  \bigg).
\end{equation}
\end{enumerate} \smallskip

\noindent These series converge uniformly on compact sets of $\C$ away from their respective singularities since the following summability condition holds
\begin{equation}\label{B/A-sum-cond}
 \sum_{n}\frac{|B(s_n)|}{|A'(s_n)|(1+s_n^2)} < \infty.
\end{equation}

\end{lemma}

\begin{proof}
The function $f(z)=B(z)/A(z)$ satisfies $\Rep [-if(z)]>0$ if $y>0$ with simple poles at the points $z=s_n$. By the Stieltjes inversion formula (see \cite[Problem 47 and Theorem 3]{B2}) 
the condition (\ref{B/A-sum-cond}) holds and 
there exists some non-positive number $p$ such that
$$
\frac{B(z)/A(z)-B(\overline{w})/A(\overline{w})}{\overline{w}-z}=p+\sum_n\frac{B(s_n)}{A'(s_n)(z-s_n)(\overline{w}-s_n)}.
$$
By the proof of \cite[Theorem 22]{B2}, if we multiply the last equality by $A(z)$, both sides would be functions in $\H(E)$.   Since $A\notin\H(E)$ we conclude that $p=0$ and this proves (1). 
To finish, we only prove (2) since (3) is analogous. For this, define 
$$
g(z)=\frac{B(s_j)}{A'(s_j)(z-s_j)} + \sum_{n\neq j} 
\frac{B(s_n)}{A'(s_n)}\bigg( \frac{1}{z-s_n} + \frac{1}{s_n-s_j}  \bigg)
$$
and note that
$$
\frac{g(z)-g(\ov w)}{\ov w -z}=\frac{B(z)/A(z)-B(\overline{w})/A(\overline{w})}{\overline{w}-z}.$$
Thus $g(z)$ differs from $B(z)/A(z)$ by a constant, that is
$$
g(z)+C=B(z)/A(z).
$$
We conclude that (for instance, via the Laurent expansions around $s_j$)
$$
C=\lim_{z\to s_j}\frac{B(z)-g(z)A(z)}{A'(s_j)(z-s_j)}= \frac{B'(s_j)}{A'(s_j)}- \frac{B(s_j)A''(s_j)}{2A'(s_j)^2}.
$$
\end{proof}

\noindent{\it Remark:} A similar lemma holds if we change $A(z)$ by $B(z)$ and $s_n$ by $t_n$.

\begin{lemma}\label{A/B-calculations}
Let $E(z)$ be a Hermite-Biehler function with no real zeros. If $B\notin\H(E)$ then

\begin{enumerate}
\item If $s_k\neq s_l$ we have
\begin{equation}\label{A/B-2-1-form}
\frac{A'(s_k)}{B(s_k)(s_k-s_l)}=\sum_n\frac{A(t_n)}{B'(t_n)(s_k-t_n)^2(s_l-t_n)}
\end{equation}
and
\begin{equation}\label{A/B-2-2-form}
-\frac{A'(s_k)}{B(s_k)(s_k-s_l)^2}-\frac{A'(s_l)}{B(s_l)(s_k-s_l)^2}=\sum_n\frac{A(t_n)}{B'(t_n)(s_k-t_n)^2(s_l-t_n)^2}.
\end{equation} \smallskip

\item For all $s_k$ we have



\begin{equation}\label{A/B-4-form}
-\frac{1}{6}\frac{\partial^3}{\partial z^3}\frac{A(z)}{B(z)}\bigg|_{z=s_k}=\sum_n\frac{A(t_n)}{B'(t_n)(s_k-t_n)^4}.
\end{equation}

\end{enumerate}

\end{lemma}

\begin{proof}
We can change the roles of $A(z)$ and $B(z)$ in Lemma \ref{int-form-B/A} to obtain
\begin{equation}\label{equa-6}
\frac{A(z)/B(z)-A(\overline{w})/B(\overline{w})}{\overline{w}-z}=
\sum_n\frac{A(t_n)}{B'(t_n)(z-t_n)(\overline{w}-t_n)}.
\end{equation}
Thus, the first part of assertion (1) follows if we differentiate the above formula with respect to $z$ 
and evaluate at the points $z=s_k$ and $\ov w=s_l$. For the second formula in (1) we  differentiate (\ref{equa-6}) with respect to $z$ and $\ov w$ and 
then evaluate at the points $z=s_k$ and $\ov w=s_l$.
For (2) we  differentiate (\ref{equa-6}) with respect to $z$ and $\ov w$ but now we evaluate at the points $z=\ov w=s_k$.
\end{proof}

Let $E(z)$ be a Hermite-Biehler function and define for every $n$ the following auxiliary functions
\begin{equation}\label{def-Pn-Qn}
P_n(z)=\frac{A(z)^2}{(z-s_n)^2} \,\,\,\, \text{ and } \,\,\,\, Q_n(z)=\frac{A(z)^2}{(z-s_n)}.
\end{equation}
These are the interpolating functions for the formula  (\ref{gen-int-form}) if we take $\alpha=-\pi/2$.
Note that $P_n,Q_n\in \H(E^2)$ for all $n$. The next lemma computes the norms and inner products associated with these functions in the 
space $\H(E^2)$ under the assumption  $AB\notin\H(E^2)$. 
We note that we can always substitute $E(z)$ by $e^{i\alpha}E(z)=A_\alpha(z)-iB_\alpha(z)$ 
for some real number $\alpha$ such that $\H(E^2)=\H(e^{2i\alpha}E^2)$ isometrically and the new functions satisfy $A_\alpha B_\alpha\notin\H(e^{2i\alpha}E^2)$.
In fact there is at most one $\alpha$ modulo $\pi/2$ such that $A_\alpha B_\alpha\in\H(E^2)$ (see the remark after Theorem \ref{gen-int-form-der}).

\begin{lemma}\label{Pn-Qn-lemma}
Let $E(z)=A(z)-iB(z)$ be a Hermite-Biehler function with no real zeros and suppose that $AB\notin \H(E^2)$. 

Then, if $s_k\neq s_l$, we have
\begin{equation} \label{prod-Pn-Pm}
\langle  P_k,P_l \rangle_{E^2} = -\bigg( \frac{A'(s_k)}{B(s_k)} + \frac{A'(s_l)}{B(s_l)}\bigg)\frac{\pi}{2(s_k-s_l)^2}
\end{equation}
and
\begin{equation} \label{prod-Qn-Qm}
\langle Q_k,Q_l\rangle_{E^2} = 0.
\end{equation}
We also have
\begin{equation} \label{norm-Pn}
\|P_k\|^2_{E^2} = -\frac{\pi}{2}\bigg( \frac{A'(s_k)^3}{B(s_k)^3} + \frac{1}{6}\frac{\partial^3}{\partial z^3}\frac{A(z)}{B(z)}\bigg|_{z=s_k}\bigg)
\end{equation}
and
\begin{equation} \label{norm-Qn}
\|Q_k\|^2_{E^2} = -\frac{\pi}{2}\frac{A'(s_k)}{B(s_k)}.
\end{equation}

\end{lemma}

\begin{proof}
Denote by $K_2(w,z)$ the reproducing kernel of $\H(E^2)$. A simple calculation would show that $K_2(w,z)=K(w,z)J(w,z)$ where $J(w,z)=2\{\ov{A(w)}A(z)+\ov{B(w)}B(z)\}$ and $K(w,z)$ is defined in (\ref{reprod-kern}). We obtain 
\begin{equation}\label{K2-K}
K_2(t_n,t_n)=2A(t_n)^3B'(t_n)/\pi \,\,\,\, \text{ and } \,\,\,\, K_2(s_n,s_n)=-2B(s_n)^3A'(s_n)/\pi.
\end{equation}
Fix $s_k\neq s_l$. Since $AB\notin\H(E^2)$ we can apply \cite[Theorem 22]{B2} to conclude that the set of functions
$$
\bigg\{\frac{A(z)B(z)}{(z-t_n)}\bigg\}\cup\bigg\{\frac{A(z)B(z)}{(z-s_n)}\bigg\} 
$$
forms an orthogonal basis of $\H(E^2)$. Note that the above functions are multiples of $K_2(t_n,z)$ and $K_2(s_n,z)$ respectively. Hence, we can calculate inner products using this orthogonal basis. We obtain
\begin{eqnarray*}
\langle P_k,P_l\rangle^2_{E^2} = \sum_n \frac{A(t_n)^4}{(t_n-s_k)^2(t_n-s_l)^2}\frac{1}{K_2(t_n,t_n)} &=& \frac{\pi}{2}\sum_n \frac{A(t_n)}{B'(t_n)(s_k-t_n)^2(s_l-t_n)^2} \\ &=& 
-\bigg( \frac{A'(s_k)}{B(s_l)} + \frac{A'(s_k)}{B(s_l)}\bigg)\frac{\pi}{2(s_k-s_l)^2},
\end{eqnarray*}
where the last equality is due to (\ref{A/B-2-2-form}). In the same way we obtain
\begin{eqnarray*}
\langle Q_k,Q_l\rangle_{E^2}= \sum_n \frac{A(t_n)^4}{(t_n-s_k)(t_n-s_l)}\frac{1}{K_2(t_n,t_n)} &=&  
\frac{\pi}{2}\sum_n \frac{A(t_n)}{B'(t_n)(t_n-s_k)(t_n-s_l)} = 0,
\end{eqnarray*}
where the last equality is due to (\ref{B/A-form-0}), since we can change the roles of $A$ 
 and $B$ in Lemma \ref{int-form-B/A}.
 
 \smallskip
 
To calculate the norms of $P_k(z)$ and $Q_k(z)$ we use the same method, but an additional term will appear due to the function $A(z)B(z)/(z-s_k)$. We obtain
\begin{eqnarray*}
\|P_k\|^2_{E^2} = -\frac{\pi}{2}\frac{A'(s_k)^3}{B(s_k)^3}+\frac{\pi}{2}\sum_n \frac{A(t_n)}{B'(t_n)(s_k-t_n)^4} 
= -\frac{\pi}{2}\bigg( \frac{A'(s_k)^3}{B(s_k)^3} + \frac{1}{6}\frac{\partial^3}{\partial z^3}\frac{A(z)}{B(z)}\bigg|_{z=s_k}\bigg),
\end{eqnarray*}
where the last equality is due to (\ref{A/B-4-form}). Analogously, by formula (\ref{B/A-form-0}), we have
$$
\|Q_k\|^2_{E^2} = \frac{\pi}{2}\sum_n \frac{A(t_n)}{B'(t_n)(s_k-t_n)^2} = 
-\frac{\pi}{2}\frac{A'(s_k)}{B(s_k)}.
$$
\end{proof}

We say that an entire function $E(z)$ is of {\it P\'olya class} if it satisfies the following conditions
\begin{enumerate}\label{Polya-class}
\item[(i)] $E(z)\neq 0$ for every $z\in\U$.
\item[(ii)] $|E^*(z)|\leq|E(z)|$ for every $z\in\U$.
\item[(iii)] $\Rep \, [iE'(z)/E(z)] \geq 0$ for every $z\in\U$.
\end{enumerate}
If $E(z)$ is of P\'olya class and real entire we say that it is of {\it Laguerre-P\'olya class}. The usual definition of the Laguerre-P\'olya class is via uniform limits on compact sets of polynomials having only real zeros, but these two definitions are equivalent (see \cite[Theorem 7 and Problems 11,12 and 13]{B2}).

If a de Branges space $\H(E^2)$ is closed by differentiation it should have some special properties. The next lemma groups together those that are relevant for our purposes.

\begin{proposition}\label{diff-lemma}
Let $\H(E^2)$ be a de Branges space closed by differentiation, then 
\begin{enumerate}
\item $E(z)$ is a function of exponential type with no real zeros.
\item The real zeros of the functions $A_\alpha(z)$ are separated and the width of separation depends only on the norm of the differentiation operator in $\H(E^2)$. 
\item The functions
$A_\alpha(z)$ are of Laguerre-P\'olya class.
\item Let $D$ denote the norm of the differentiation operator in $\H(E^2)$. Then for every real number $\alpha$ we have
\begin{equation}\label{Sec2_Rev_eqD}
A_\alpha''(s)^2+4A_\alpha'(s)^2 \leq  (D^2+D^4)B_{\alpha}(s)^2,
\end{equation}
whenever $A_\alpha(s)=0$.
\item The function $\p'(x)$ is bounded.
\end{enumerate}
\end{proposition}

\begin{proof}
 First we prove (1). If $F\in\H(E^2)$ then for any $w\in\C$ we have
 $$
 |F(z)| \leq \sum_{n\geq 0} \frac{|F^{(n)}(w)|}{n!}|z-w|^n \leq \|F\|_{E^2} K_2(w,w)^{1/2} \sum_{n\geq 0} \frac{D^n|z-w|^n}{n!}=\|F\|_{E^2}K_2(w,w)^{1/2}e^{D|z-w|}\,,
 $$
where we have used (\ref{swartz-ineq}) and $D$ denotes the norm of the differentiation operator. We conclude that every function $F\in\H(E^2)$ is of exponential type at most $D$. Fix a function $F\in\H(E^2)$ with $F(i)\neq 0$. We conclude that $G(z)=[F(i)E(z)^2-E(i)^2F(z)]/(z-i)$ belongs to $\H(E^2)$
 and 
 $$
 E(z)^2=[(z-i)G(z)+E(i)^2F(z)]/F(i).
 $$
Hence, $E(z)$ is of exponential type at most $D/2$. $E(z)$ cannot have real zeros since the differentiation reduces the order of the zeros (this argument is due to A. Baranov see \cite{Bar2}).
 
\smallskip

Now we prove (4). Since $(e^{i\alpha}E(z))^2$ generates the same space that $E(z)^2$ generates, we can assume that $\alpha=0$. 
We have the following Taylor's expansion for the function $Q_n(z)$
 $$
 Q_n(z)=A'(s_n)^2(z-s_n) + A'(s_n)A''(s_n)(z-s_n)^2 + ...
 $$
Letting $K_2(w,z)$ be the reproducing kernel of $\H(E^2)$ and using the Cauchy-Schwarz inequality, we obtain
 $$
 A'(s_n)^4 + 4A'(s_n)^2A''(s_n)^2=|Q_n'(s_n)|^2+|Q_n''(s_n)|^2\leq (D^2+D^4)\|Q_n\|_{E^2}^2K_2(s_n,s_n).
 $$
 Since 
 $$
 K_2(s_n,s_n)=-2A'(s_n)B(s_n)^3/\pi
 $$
 and
 \begin{equation}\label{est-Qn-norm}
 \|Q_n\|_{E^2}^2\leq \|A(z)/(z-s_n)\|^2_E=-\pi [A'(s_n)/B(s_n)]
 \end{equation}
we obtain the desired inequality \eqref{Sec2_Rev_eqD}.

\smallskip
 
Now we prove (3). First assume that $\alpha=0$ and $A\notin \H(E)$.
Take $F\in\H(E^2)$ such that $F(0)=1$ and write $a=E(0)$. We conclude that
$$
 \frac{\partial}{\partial z}\frac{E(z)^2-F(z)a^2}{z}=[2E'(z)E(z)-F'(z)a^2]/z - [E(z)^2-F(z)a^2]/z^2
$$
belongs to the space $\H(E^2)$. Using (\ref{mean-type-relations}) we conclude that $E'(z)/E(z)$ is of bounded type in $\U$ with non-positive mean type. Also,
 $$
 \int_\R |E'(t)/E(t)|^2\frac{\dt}{1+t^2}<\infty.
$$
Applying the same argument with $E^*(z)$ we obtain that
$$
 \frac{\partial}{\partial z}\frac{E(z)E^*(z)-F(z)|a|^2}{z}=[E'(z)E^*(z)+E(z)E'^*(z)-F'(z)|a|^2]/z - [E(z)E^*(z)-F(z)|a|^2]/z
 $$
belongs to the space $\H(E^2)$, hence $E'^*(z)/E(z)$ is of bounded type in $\U$ with non-positive mean type. We conclude that $A'(z)/E(z)$ is of bounded type in $\U$ with non-positive mean type.  
 
\smallskip

Now take $b\in\R$ such that $A(b)\neq 0$ and $F\in\H(E)$ 
with $F(b)=1$. Then $[A'(z)-A'(b)F(z)]/(z-b)$ belongs to $\H(E)$ and, since $A\notin\H(E)$, we can apply \cite[Theorem 22]{B2} to obtain
 $$
 \frac{A'(z)-A'(b)F(z)}{A(z)(z-b)}=\sum_m  \frac{1-A'(b)F(s_m)/A'(s_m)}{(s_m-b)(z-s_m)}.
 $$
 By the same theorem we have 
 $$
 F(z)/A(z)=\sum_m\frac{F(s_m)}{A'(s_m)(z-s_m)}.
 $$
We conclude that 
 $$
\Rep \,\, i\frac{A'(z)}{A(z)} = y\sum_m \frac{1}{|z-s_m|^2} > 0
 $$
 for $y>0$. Hence $A(z)$ is of P\'olya class (see \cite[Section 7]{B2}). Since $A(z)$ is real entire, it belongs to the Laguerre-P\'olya class.
 A similar argument would show that $A_\alpha(z)$ is of Laguerre-P\'olya class whenever $A_\alpha\notin\H(E)$. Since the Laguerre-P\'olya class is closed by pointwise limits and there exists at most one $\alpha$ modulo $\pi/2$ such that $A_\alpha\in\H(E)$, item (3) follows.
 
 \smallskip
 
We now prove (2). Assume first that $AB \notin \H(E^2)$. By inequality (\ref{est-Qn-norm}) we get, for all $m$ and $n$, 
\begin{equation}\label{A-ineq-cor}
 A(t_m)^4/(t_m-s_n)^2 =|Q_n(t_m)|^2\leq \|Q_n\|^2_{E^2}K_2(t_m,t_m) = -2[A'(s_n)/B(s_n)]B'(t_m)A(t_m)^3.
\end{equation}
Recalling that $B_{\alpha-\pi/2}=A_\alpha$, by item (4) we obtain 
\begin{equation}\label{Sec2_Rev_eq_t_n_s_m}
(t_m-s_n)^{-2}\<_D 1\,,
\end{equation}
which proves item (2), since the points $\{t_n\}$ and $\{s_n\}$ are interlaced.

\smallskip
 
Finally, for item (5), note that
 $$
 \p'(s_n)=- [A'(s_n)/B(s_n)],
 $$
which is bounded by item (4). In general, if we take a real point $s$ such that $\p(s)\equiv \alpha-\pi/2 \ (\!\!\!\!\mod \pi)$ then 
 $$
 \p'(s)=  -[A'_\alpha(s)/B_\alpha(s)] \<_D 1.
 $$
\end{proof}
 
 \noindent {\it Remark:} In \cite[Section 4.1]{Bar} A. Baranov constructed spaces $\H(E)$ that are closed by differentiation, but $\p'(x)$ is unbounded. Thus, for a space $\H(E^2)$ to be closed by differentiation we have to require stronger restrictions on the function $E(z)$. For instance, the boundedness of $\p'(x)$ will play an important role in the proof of Theorem $\ref{gen-int-form-der}$, since it implies that the points of interpolation $\Tau(\alpha)$ are separated.
 
 \smallskip

For the sake of completeness we state here a result about Hilbert-type inequalities proved in \cite[Corollary 22]{CLV}.

\begin{proposition}\label{hilb-type-ineq}
Let $\xi_1, \xi_2,..., \xi_N$ be real numbers such that $0<\sigma \leq |\xi_n-\xi_m|$ whenever $m \neq n$. Let $a_1, a_2,..., a_N$ be complex numbers. Then
$$
-\frac{\pi^2}{6\sigma^2}\sum_{n=1}^N |a_n|^2 \leq \sum_{\stackrel{m,n=1}{m\neq n}}^N\frac{a_n\ov a_m}{(\xi_n-\xi_m)^2}\leq \frac{\pi^2}{3\sigma^2}\sum_{n=1}^N|a_n|^2.
$$
The constants appearing in these inequalities are the best possible.
\end{proposition}

\subsection{Proof of Theorem \ref{gen-int-form-der} - The case $p=2$}\label{proof-p=2}

The idea of the proof is to show that (\ref{gen-int-form})
holds for a dense set of functions in $\H(E^2)$ and then argue that we can interchange limits and summation. In fact we will show convergence of the formula in the space $\H(E^2)$, which implies convergence on compact sets of $\C$. \smallskip

First of all, we can assume $\alpha=-\pi/2$ which is no restriction since $\H(E^2)=\H(e^{2i\alpha}E^2)$ isometrically. Also, note that $B_{-\pi/2}(z)=A(z)$ and $A_{-\pi/2}(z)=-B(z)$. We will denote by $D$ the norm of the differentiation operator in $\H(E^2)$. By the hypothesis of the theorem there exists a number $\delta>0$ such that 
\begin{equation}\label{magic-cond}
|A'(s_n)/B(s_n)|=\p'(s_n)\geq\delta \,\,\, {\text{ for all }} n .
\end{equation}
We divide the proof in a few steps. \smallskip

\noindent {\it \large Step 1}. We show that the quantities in (\ref{norm-Pn}) and (\ref{norm-Qn}) are uniformly bounded. By (\ref{norm-Qn}) and Proposition \ref{diff-lemma} item (4), we have
\begin{equation}\label{norm-Qn-estimate}
\|Q_k\|^2_{E^2} = -\frac{\pi}{2}\frac{A'(s_k)}{B(s_k)} \<_D 1.
\end{equation}
By (\ref{norm-Pn}) we have 
$$
\|P_k\|^2_{E^2} = -\frac{\pi}{2}\bigg( \frac{A'(s_k)^3}{B(s_k)^3} + 
\frac{1}{6}\bigg[\frac{A'''(s_k)}{B(s_k)}-3\frac{A''(s_k)B'(s_k)}{B(s_k)^2}-3\frac{A'(s_k)B''(s_k)}{B(s_k)^2}
+6\frac{A'(s_k)B'(s_k)^2}{B(s_k)^3}\bigg]\bigg).
$$
Again, by Proposition \ref{diff-lemma} item (4), we obtain
\begin{equation}\label{norm-Pn-ineq}
 \|P_k\|^2_{E^2} \<_D 1 + \bigg|\frac{A'''(s_k)}{B(s_k)}\bigg| + \bigg|\frac{B'(s_k)}{B(s_k)}\bigg| 
+ \bigg|\frac{B''(s_k)}{B(s_k)}\bigg| + \bigg|\frac{B'(s_k)}{B(s_k)}\bigg|^2.
\end{equation}
We claim that each quantity appearing on the right hand side of the last inequality is bounded independently of $s_k$. By definition (\ref{def-Pn-Qn}), identities (\ref{norm-Qn}) and (\ref{K2-K}) we have
\begin{eqnarray*}
|2A'''(s_k)A'(s_k) + 3A''(s_k)^2/2|^2&=&|Q_k'''(s_k)|^2 \\ &\leq& D^6 \|Q_k\|^2_{E^2} K_2(s_k,s_k) \\&=&D^6 |A'(s_k)^2B(s_k)^2|.
\end{eqnarray*}
Hence, by Proposition \ref{diff-lemma} item (4) and hypothesis (\ref{magic-cond}) we obtain
\begin{equation}\label{norm-Pn-ineq-1}
 \bigg|\frac{A'''(s_k)}{B(s_k)}\bigg| \<_{D,\delta}1.
\end{equation}
If we write $R_k(z)=A(z)B(z)/(z-s_k)$ for every $k$, we obtain 
\begin{eqnarray*}
|R_k'(s_k)|^2+|R_k''(s_k)|^2 &\leq &(D^2 + D^4)\|R_k\|^2_{E^2} K_2(s_k,s_k)\\ & \leq&  (D^2 + D^4)\|A(z)/(z-s_k)\|^2_E K_2(s_k,s_k),
\end{eqnarray*}
which is equivalent to
\begin{eqnarray*}
 |A''(s_k)B(s_k)/2+A'(s_k)B'(s_k)|^2 & + &| A'''(s_k)B(s_k)/3+A''(s_k)B'(s_k)+A'(s_k)B''(s_k)|^2  \\
 &\leq &  2(D^2+D^4)|A'(s_k)B(s_k)|^2. 
\end{eqnarray*}
Dividing both sides by $|A'(s_k)B(s_k)|^2$ we obtain
\begin{equation}\label{equa-9}
\bigg|\frac{A''(s_k)}{2A'(s_k)}+\frac{B'(s_k)}{B(s_k)}\bigg| \<_D 1 
\end{equation}
and 
\begin{equation}\label{equa-10}
\bigg|\frac{A'''(s_k)}{3A'(s_k)}+\frac{A''(s_k)B'(s_k)}{A'(s_k)B(s_k)}+\frac{B''(s_k)}{B(s_k)}\bigg| \<_D 1.
\end{equation}
Using (\ref{equa-9}) we obtain
\begin{equation}\label{norm-Pn-ineq-2}
\bigg|\frac{B'(s_k)}{B(s_k)} \bigg| \<_D 1+\bigg|\frac{A''(s_k)}{A'(s_k)}\bigg| = 1+
\bigg|\frac{A''(s_k)/B(s_k)}{A'(s_k)/B(s_k)}\bigg| \<_{D,\delta}1,
\end{equation}
where the last inequality is due to Proposition \ref{diff-lemma} item (4) and (\ref{magic-cond}). Using (\ref{equa-10}) we obtain
\begin{equation*}
\bigg|\frac{B''(s_k)}{B(s_k)} \bigg| \<_{D} 1+\bigg|\frac{A'''(s_k)/B(s_k)}{A'(s_k)/B(s_k)}\bigg|+\bigg|\frac{A''(s_k)B'(s_k)}{A'(s_k)B(s_k)}\bigg|.
\end{equation*}
Hence, by (\ref{magic-cond}), (\ref{norm-Pn-ineq-1}) and (\ref{norm-Pn-ineq-2}) we obtain
\begin{equation}\label{norm-Pn-ineq-3}
\bigg|\frac{B''(s_k)}{B(s_k)} \bigg|\<_{D,\delta} 1.
\end{equation}
Thus, by (\ref{norm-Pn-ineq}), (\ref{norm-Pn-ineq-1}), (\ref{norm-Pn-ineq-2}) and (\ref{norm-Pn-ineq-3}) we obtain
$$
\|P_k\|^2_{E^2} \<_{D,\delta} 1.
$$

\smallskip

\noindent {\it \large Step 2.} Since $E(z)^2=A(z)^2-B(z)^2 - i2A(z)B(z)$ and, by hypothesis, $AB\notin\H(E^2)$ we conclude that $A,B\notin\H(E)$ and the functions
$$
\{A(z)B(z)/(z-s_j)\} \cup\{A(z)B(z)/(z-t_j)\}
$$
form an orthogonal basis of $\H(E^2)$. We show that formula (\ref{gen-int-form}) holds for any of these functions, 
hence it holds for any finite linear combination of them. If we put $F(z)=A(z)B(z)/(z-s_j)$ on the right hand side of formula (\ref{gen-int-form})
we obtain
$$
A(z)^2\bigg[ \frac{B(s_j)}{A'(s_j)(z-s_j)^2} - \frac{B(s_j)A''(s_j)}{2A'(s_j)^2(z-s_j)} + \frac{B'(s_j)}{A'(s_j)(z-s_j)} 
+\sum_{n\neq j} \frac{B(s_n)}{A'(s_n)(z-s_n)(s_n-s_j)} \bigg].
$$
This is equal to $A(z)B(z)/(z-s_j)$ by Lemma \ref{int-form-B/A} formula (\ref{B/A-form-1}). A similar argument would show that  formula (\ref{gen-int-form}) holds for $F(z)=A(z)B(z)/(z-t_j)$, but now using 
Lemma \ref{int-form-B/A} formula (\ref{B/A-form-2}). 

\smallskip

\noindent {\it \large Step 3}. Now we prove that formula (\ref{gen-int-form}) converges in the norm of $\H(E^2)$ for every $F\in\H(E^2)$. Since $AB\notin\H(E^2)$, by (\ref{norm-basis}) and (\ref{K2-K}), if $F\in\H(E^2)$ we have
\begin{equation}\label{norm-E2}
 \|F\|^2_{E^2} = \sum_n \bigg\{\frac{|F(s_n)|^2}{K_2(s_n,s_n)} + \frac{|F(t_n)|^2}{K_2(t_n,t_n)}\bigg\} = 
\frac{\pi}{2}\sum_n \bigg\{\frac{|F(s_n)|^2}{|A'(s_n)B(s_n)^3|} + \frac{|F(t_n)|^2}{B'(t_n)A(t_n)^3}\bigg\}.
\end{equation}
Hence, to prove the convergence of formula (\ref{gen-int-form}) in the space $\H(E^2)$, it is sufficient to show the following inequality
\begin{align}\label{vital-ineq}
\begin{split}
 \bigg\|\sum_{n\in I} \bigg\{\frac{z_n}{A'(s_n)^2}P_n(z)+ 
\frac{w_n}{A'(s_n)^2}Q_n(z)-\frac{z_nA''(s_n)}{A'(s_n)^3}Q_n(z) \bigg\}\bigg\|^2_{E^2}  \\ \<_{D,\delta} \sum_{n\in I} \bigg\{
\frac{|z_n|^2}{|A'(s_n)B(s_n)^3|} + \frac{|w_n|^2}{|A'(s_n)B(s_n)^3|}\bigg\}
\end{split}
 \end{align}
for every finite set $I\subset\Z$ and complex numbers $\{z_n,w_n\}_{n\in I}$. This would show, together with (\ref{norm-E2}), that the partial sums of formula (\ref{gen-int-form}) form a Cauchy sequence in the norm $\|\cdot\|_{E^2}$ for all $F\in\H(E^2)$.
\smallskip

By Lemma \ref{Pn-Qn-lemma} formula (\ref{prod-Qn-Qm}) the functions $\{Q_n(z)\}$ are orthogonal, thus
\begin{eqnarray*}
 \bigg\|\sum_{n\in I}  
\frac{w_n}{A'(s_n)^2}Q_n(z)-\frac{z_nA''(s_n)}{A'(s_n)^3}Q_n(z) \bigg\|^2_{E^2} &\<_D& 
\sum_{n\in I}  \frac{|w_n|^2}{|A'(s_n)|^4} + \frac{|z_nA''(s_n)|^2}{|A'(s_n)|^6}  \\ &\<_{D,\delta}& 
\sum_{n\in I}  \frac{|w_n|^2}{|A'(s_n)B(s_n)^3|} + \frac{|z_n|^2}{|A'(s_n)B(s_n)^3|}\,,
\end{eqnarray*}
where the first inequality is due to orthogonality and estimate (\ref{norm-Qn-estimate}) of Step 1. The last inequality is due to (\ref{magic-cond}) and Proposition \ref{diff-lemma} item (4). Analogously, by Lemma \ref{Pn-Qn-lemma} formula (\ref{prod-Pn-Pm}) and Step 1, we obtain
\begin{eqnarray*}
\bigg\|\sum_{n\in I} \frac{z_n}{A'(s_n)^2}P_n(z)\bigg\|^2_{E^2} & =&
\sum_{n,m\in I} \frac{z_n \ov z_m}{A'(s_n)^2A'(s_m)^2}\langle P_n,P_m \rangle_{E^2} \\ 
&\<_{D}& 
\sum_{\{n\neq m\}\subset I} \frac{|z_nz_m|}{A'(s_n)^2A'(s_m)^2(s_n-s_m)^2} + \sum_{n\in I}\frac{|z_n|^2}{A'(s_n)^4}.
\end{eqnarray*}
The first term on the right hand side of the last inequality is in the form of a Hilbert-type sum as in Proposition \ref{hilb-type-ineq}, at the points $\xi_n=s_n$ and $a_n=|z_n|/A'(s_n)^2$. By Proposition \ref{diff-lemma}, the zeros of $A(z)$ are separated with width of separation depending only on $D$. Hence we can apply Proposition \ref{hilb-type-ineq} to obtain
$$
\bigg\|\sum_{n\in I} \frac{z_n}{A'(s_n)^2}P_n(z)\bigg\|^2_{E^2}  \<_D
\sum_{n\in I}\frac{|z_n|^2}{A'(s_n)^4} \<_{D,\delta} \sum_{n\in I}\frac{|z_n|^2}{|A'(s_n)B(s_n)^3|}.
$$
This proves the desired inequality (\ref{vital-ineq}). Also note that if we define
\begin{equation}\label{F0-def}
F_0(z)=\lim_{N\to\infty} A(z)^2\sum_{|k|\leq N}\bigg\{\frac{F(s_k)}{A'(s_k)^2(z-s_k)^2}+ 
\frac{F'(s_k)A'(s_k)-F(s_k)A''(s_k)}{A'(s_k)^2(z-s_k)}  \bigg\}\,,
\end{equation}
then by (\ref{norm-E2}) and (\ref{vital-ineq}) we have
\begin{equation}\label{vital-ineq-2}
\|F_0\|^2_{E^2} \<_{D,\delta} \, \sum_k \frac{|F(s_k)|^2+|F'(s_k)|^2}{K_2(s_k,s_k)} \leq (1+D^2)\|F\|^2_{E^2}.
\end{equation}

\smallskip

\noindent {\it \large Step 4}. Now we finish the proof. Take $F\in\H(E^2)$ and denote by $F_0\in\H(E^2)$ the function given by the formula (\ref{F0-def}). Note that the $F_0(z)$ is well defined due to Step 3. We claim that $F=F_0$. Given $\varepsilon>0$, by Steps 2 and 3 there exists a function $G\in\H(E^2)$ 
such that the formula holds and $\|F-G\|_{E^2}<\varepsilon$, which implies $\|F'-G'\|_{E^2}<D\varepsilon$. We obtain
\begin{eqnarray*}
 \|F-F_0\|^2_{E^2} < 2\varepsilon^2 + 2\|F_0-G\|^2_{E^2} & \<_{D,\delta} & 2\varepsilon^2 + \sum_{n} 
\frac{|F(s_n)-G(s_n)|^2}{|A'(s_n)B(s_n)^3|} + \frac{|F'(s_n)-G'(s_n)|^2}{|A'(s_n)B(s_n)^3|} 
\\ & \leq &  2\varepsilon^2 + \|F-G\|^2_{E^2} + \|F'-G'\|^2_{E^2} \\ &<& 3\varepsilon^2 + D^2\varepsilon^2 ,
\end{eqnarray*}
where the second inequality is due to (\ref{vital-ineq-2}) and the third due to (\ref{norm-E2}). Since $\varepsilon>0$ is arbitrary, we conclude the proof.

\subsection{Proof of Corollary \ref{cor}}

By the Plancherel-P\'olya Theorem (see \cite{PP}) $\PW(a,2)$ is closed by differentiation. Denote by $K_2(w,z)$ the reproducing kernel of $\H(E^2)$ and note that $H(w,z)=\frac{\sin(a(z-\ov w))}{\pi(z-\ov w)}$ is the reproducing kernel of $\PW(a,2)$. Since $\PW(a,2)=\H(E^2)$ as sets, by the Closed Graph Theorem there exists a constant $C>0$ such that 
\begin{equation}\label{norm-equiv}
C^{-1}\|F\|_{L^2(\R)}\leq \|F\|_{E^2}\leq C\|F\|_{L^2(\R)},
\end{equation}
for every $F\in\PW(a,2)$. The reproducing kernel property implies that
$$
K_2(w,w)=\sup \, \{|F(w)|^2 : F\in\PW(a,2), \,\,\,\, \|F\|_{E^2}\leq 1\}
$$
and
$$
\frac{\sin(a(w-\ov w))}{\pi(w-\ov w)}=\sup \, \{|F(w)|^2 : F\in\PW(a,2), \,\,\,\, \|F\|_{L^2(\R)}\leq 1\}
$$
for every $w\in\C$. We conclude that 
\begin{equation}\label{est-K2-cor}
K_2(t,t)\simeq_{a,C} 1
\end{equation}
for all real $t$. Since $|E(t)|^4\p'(t)=\tfrac{\pi}{2} K_2(t,t)$, and $|A(t)|\leq M$ whenever $B(t)=0$, we conclude that 
$$
1 \<_{C,a,M} \p'(t)
$$ 
whenever $B(t)=0$. 

We claim that $AB\notin\PW(a,2)$. Since $\H(E^2)=\PW(a,2)$ we easily obtain that $\H(E_a^2)=\PW(\pi,2)$ where $E_a(z)=E(\tfrac{\pi}{a}z)$. Since $\PW(\pi,2)$ is closed by differentiation, by Proposition \ref{diff-lemma}, the real zeros of $L(z)=A(\tfrac{\pi}{a}z)B(\tfrac{\pi}{a}z)$ are separated. Hence, we can apply \cite[Theorem 1]{OS} to conclude that the sequence $\{t\in\R: L(t)=0\}$ is sampling for $\PW(\pi,2)$, that is
$$
\int_\R|F(x)|^2\dx \simeq \sum_{L(t)=0}|F(t)|^2
$$
for every $F\in\PW(\pi,2)$. Thus $AB\notin\PW(a,2)$, otherwise $L(z)$ would belong to $\PW(\pi,2)$ and have zero norm, a contradiction. We conclude that all the conditions of Theorem \ref{gen-int-form-der} are satisfied for $\H(E^2)$ and $\alpha=0$. By the interpolation formula (\ref{gen-int-form}), the proof of Theorem \ref{gen-int-form-der} and estimates (\ref{vital-ineq-2}), (\ref{norm-equiv}) and (\ref{est-K2-cor}), the corollary easily follows.
\medskip

\noindent {\it Remark:} By Proposition \ref{diff-lemma} item (5) and inequalities (\ref{A-ineq-cor}) and (\ref{est-K2-cor}) we conclude that 
$$
|A(t)|^4 \leq \pi|t-s|^2\p'(s)K_2(t,t) \< |t-s|^2,
$$
whenever $B(t)=0$ and $A(s)=0$. Hence, the condition 
$$
\sup_{B(t)=0}\inf_{A(s)=0}|t-s|<\infty
$$
ensures the existence of a number $M>0$ such that $|A(t)|\leq M$ whenever $B(t)=0$.

\section{$L^p$ de Branges spaces}\label{Lp-form}

\subsection{Preliminaries} Recall that we denote by $\U$ the open upper half-plane. For a given $p\in(0,\infty]$ we define the Hardy space $H^p(\U)$ as the space of functions $F:\U\to\C$ analytic in $\U$ such that 
$$
\sup_{y>0}\|F(\cdot+iy)\|_p <\infty\,,
$$
where $\|\cdot\|_p$ stands for the standard $L^p(\R,\dx)$-norm. In the case $p\in[1,\infty]$ it can be proven that for every $F\in H^p(\U)$ the limit
$$
F(x)=\lim_{{y\to 0}}F(x+iy)
$$
exists for almost every real $x$ and defines a function in $L^p(\R,\dx)$. Moreover, the following Poisson representation holds
 \begin{equation*}\label{poisson-rep}
  \Rep \, F(z) = \frac{y}{\pi}\int_\R \frac{ \Rep \, F(t)}{(x-t)^2+y^2}\,\dt.
 \end{equation*}
 Using this representation and Young's inequality for convolutions, one can deduce that $\sup_{y>0}\|F(\cdot+iy)\|_p=\|F\|_p$ and $H^p(\U)$ is a Banach space for $p\geq 1$. All these facts are contained in \cite{ABR}.
 \smallskip
 

The next proposition provides a different definition of the spaces $\H^p(E)$. 

\begin{proposition}
Let $F(z)$ be an analytic function on the upper half-plane that has a continuous extension to the closed upper half-plane. The following are equivalent:

\begin{enumerate}
 \item $\sup_{y>0}\|F(\cdot+iy)\|_p<\infty$
 \item $F(z)$ is of bounded type in $\U$ with non-positive mean type and
 $$
 \|F\|_p < \infty.
 $$
\end{enumerate}
\end{proposition}
\begin{proof}
 First we prove $(2)\implies(1)$. Since $F(z)$ is of bounded type with non-positive mean type we have (see \cite[Problem 27]{B2})
 $$
 \log |F(z)|\leq \frac{y}{\pi}\int_\R \frac{ \log |F(t)|}{(x-t)^2+y^2}\dt.
 $$
 Jensen's inequality implies that
 $$
 |F(z)|\leq \frac{y}{\pi}\int_\R \frac{ |F(t)|}{(x-t)^2+y^2}\dt.
 $$
 Applying Young's inequality for convolutions and Fatou's lemmma we conclude that
 $$
 \sup_{y>0}\|F(\cdot+iy)\|_p = \|F\|_p.
 $$
 
 For $(1)\implies(2)$ we use the fact that
 $$
  \Rep \, F(z) = \frac{y}{\pi}\int_\R \frac{ \Rep \, F(t)}{(x-t)^2+y^2}\,\dt
 $$
 in $\U$ (see \cite[Theorem 7.14]{ABR}). Write  $\Rep\, F(t)=g(t)-h(t)$, where $g(t)=\max\{\Rep\, F(t),0\}$ and  $h(t)=\max\{-\Rep\, F(t),0\}$. Let $G(z)$ and $H(z)$ be analytic functions in $\U$ such that 
 $$
\Rep \,G(z)=\frac{y}{\pi}\int_\R \frac{ g(t)}{(x-t)^2+y^2}\dt
$$ 
and 
$$
\Rep \,H(z)=\frac{y}{\pi}\int_\R \frac{ h(t)}{(x-t)^2+y^2}\dt.
$$
Since $\Rep \, H(z)>0$ and $\Rep \, G(z)>0$ in $\U$, we conclude that $G(z)$ and $H(z)$ are of bounded type with non-positive mean type (see \cite[Problem 20]{B2}). Since $F(z)$ differs from $G(z)-H(z)$ by a constant, we conclude that $F(z)$ is of bounded type with non-positive mean type.
 \end{proof}

\noindent {\it Remark:} The above proposition implies that $F\in\H^p(E)$ if and only if 
$$
\sup_{y\in\R}\|F(\cdot+iy)/E(\cdot+i|y|)\|_p<\infty,
$$
or equivalently, if $F/E$ and $F^*/E$ belong to $H^p(\U)$. It can be proven, using the completeness of Hardy spaces and the reproducing kernel property (\ref{repord-kern-id}) that the spaces $\H^p(E)$ are Banach spaces for $p\geq 1$.

\smallskip

The next three lemmas are technical tools needed for the full proof of Theorem \ref{gen-int-form-der}.
\begin{lemma}\label{concatenation-lemma}
Let $E(z)$ be a Hermite-Biehler function such that $\p'(x)$ is bounded. Then $\H^p(E)\subset\H^q(E)$ continuously if $1\leq p<q<\infty$.
\end{lemma}
\begin{proof}
First we show that $\H^p(E)\subset \H^\infty(E)$ if $p\in[1,2]$.

\smallskip

Recall that $\p'(x)=\pi K(x,x)/|E(x)|^2$ and denote by $C$ its supremum. By the reproducing kernel property  we obtain
$$
\|K(t,\cdot)\|^2_{E}=K(t,t)\leq C|E(t)|^2/\pi
$$
for all real $t$. In the same way, noting that
$$
\|K(t,\cdot)\|^2_{E,\infty}=\sup_{x\in\R}{\bigg|\frac{K(t,x)}{E(x)}\bigg|^2},
$$
and $K(t,x)^2 \leq K(x,x)K(t,t)$, we conclude that 
$$
\|K(t,\cdot)\|^2_{E,\infty} \leq C^2|E(t)|^2/\pi^2.
$$

\noindent Hence, we obtain that for all $q\in[2,\infty]$
\begin{equation}\label{norm-kern-bound}
\|K(t,\cdot)\|_{E,q} \leq (C/\pi)^{1-1/q}|E(t)|.
\end{equation}

\noindent If $p\in[1,2]$ and $F\in\H^p(E)$, then for all $t\in\R$
$$
|F(t)/E(t)| \leq \|F\|_{E,p}\|K(t,\cdot)\|_{E,p'}/|E(t)| \leq \|F\|_{E,p}(C/\pi)^{1/p}.
$$
This implies the proposed inclusions for $1\leq p<q\leq \infty$ and $p\leq 2$. By \cite[Proposition 1.1]{Bar2} and {\cite[Lemma 4.2]{Co}} the dual space of $\H^p(E)$ can be identified with $\H^{p'}(E)$ if $1<p<\infty$. This implies the remaining inclusions. Since convergence in the space implies convergence on compacts sets of $\C$ we conclude that the identity map from $\H^p(E)$ to $\H^q(E)$ is closed, hence continuous by the Closed Graph Theorem.
\end{proof}

\begin{lemma}\label{tech-lemma-density}
Let $E(z)$ be a Hermite-Biehler function such that $\p'(x)$ is bounded. Let $\alpha\in\R$ be such that $B_\alpha\notin\H(E)$. Then the linear span of the following set of functions 
$$
\{B_\alpha(z)/(z-t)\}_{t\in\Tau(\alpha)}
$$
is dense in $\H^p(E)$ for every $p\in[2,\infty)$.
\end{lemma}

\begin{proof}
\noindent Denote by $S$ the closure of this span in $\H^p(E)$. Suppose by contradiction that $S\neq \H^p(E)$. By the Hahn-Banach Theorem there exists a non-zero functional $\Lambda\in\H^p(E)'$ that vanishes on $S$. Since $\H^p(E)'=\H^{p'}(E)$, we conclude that $\Lambda=\Lambda(z)$ is an entire function that belongs to $\H^{p'}(E)$ and
$$
\Lambda(t)=\langle \Lambda,K(t,\cdot)\rangle_E =0
$$
for every $t\in\Tau(\alpha)$, since $K(t,z)$ is a multiple of $B_\alpha(z)/(z-t)$ for every $t\in\Tau(\alpha)$. 
By  Lemma \ref{concatenation-lemma}, $\H^{p'}(E)\subset \H(E)$. Since $B_\alpha\notin\H(E)$, the set $\{B_\alpha(z)/(z-t)\}_{t\in\Tau(\alpha)}$ forms an orthogonal basis of $\H(E)$ and we conclude by (\ref{norm-basis}) that $\Lambda\equiv 0$, a contradiction.
\end{proof}

\begin{lemma}\label{tech-lemma-p>2}
Let $E(z)$ be a Hermite-Biehler function such that $\p'(x)\leq C$ for every real $x$. Then for every $p\in [1,2]$, every real number $\alpha$ and every $F\in\H^p(E)$ we have
\begin{equation}\label{Hp-sum-strange}
\sum_{t\in\Tau(\alpha)} \frac{|F(t)|}{(1+|t|)K(t,t)^{1/2}} \<_{C,p} \|F\|_{E,p}.
\end{equation}
This inequality is also valid for $p\in(2,\infty)$ if we additionally assume that $v(E^*/E)<0$.

\end{lemma}

\begin{proof}
By hypothesis, if $t<t'$ are two consecutive points in $\Tau(\alpha)$, we have
$$
\pi = (t'-t)\p'(r) 
$$
for some $r\in(t,t')$. We conclude that the points $\Tau(\alpha)$ are $\pi/C$-separated.
We divide the proof in steps.

\smallskip

\noindent {\it Step 1.} The inequality (\ref{Hp-sum-strange}) is valid for $p\in[1,2]$.
\smallskip

By Lemma \ref{concatenation-lemma} we have $\H^p(E)\subset\H(E)$ continuously, thus the case $p<2$ follows directly from the case $p=2$. Let $F\in\H(E)$, by the Cauchy-Schwarz inequality we have
$$
\sum_{t\in\Tau(\alpha)} \frac{|F(t)|}{(1+|t|)K(t,t)^{1/2}} \leq \bigg(\sum_{t\in\Tau(\alpha)} \frac{|F(t)|^2}{K(t,t)} \bigg)^{1/2}\bigg(\sum_{t\in\Tau(\alpha)} (1+|t|)^{-2}\bigg)^{1/2} \<_{C} \|F\|_{E},
$$
where the last inequality is due to (\ref{norm-basis}) and the separability of $\Tau(\alpha)$.
\smallskip

\noindent  {\it Step 2.} The case $p>2$.
\smallskip

By hypothesis, let $v(E^*/E)=-2a<0$. Fix a real number $\nu$ such that $\nu\in(-1/p,0)$.
Let $E_\nu(z)$ be the function defined in Section \ref{hom-space} and  define the operator $\LL :\H^p(E)\to\H(E^2)$ by $\LL F(z)=e^{-iaz}E_\nu(az)E^*(z)F(z)$. By the properties described in Section \ref{hom-space} we have
\begin{enumerate}
\item[(i)] $v(E_\nu^*)\leq v(E_\nu)=\tau(E_\nu)=1$;
\item[(ii)] $|E_\nu(t)| \simeq 1/|t|^{\nu+1/2} \,\,\,$, for $|t|\geq 1$.
\end{enumerate}
Hence, if $G(z)=\LL F(z)$ we obtain
$$
v(G/E^2)= v(F/E) + v(E^*/E)+ v(E_\nu(az))+ v(e^{-iaz})\leq 0-2a +a  +a =0
$$
and 
$$
v(G^*/E^2)=v(F^*/E) + v(E_\nu^*(az))+ v(e^{iaz})\leq 0+ a - a =0.
$$
We also have
\begin{equation}\label{equa-4}
\int_\R|G(t)/E(t)^2|^2\,\dt \leq  \bigg(\int_\R|E_\nu(at)|^q\dt\bigg)^{2/q} \|F\|^2_{E,p}\<_{a,p}\|F\|^2_{E,p},
\end{equation}
where $1/2=1/q+1/p$. Note that $q>2$ and $q(\nu+1/2)>1$. We conclude that the operator $\LL$ is well-defined and continuous. Denoting by $K_2(w,z)$ the reproducing kernel of $\H(E^2)$ and $K(w,z)$ the reproducing kernel of $\H(E)$ we obtain $K_2(t,t)=2|E(t)|^2K(t,t)$. We have
\begin{align*}
\sum_{t\in\Tau(\alpha)} \frac{|F(t)|}{(1+|t|)K(t,t)^{1/2}}&=\sum_{t\in\Tau(\alpha)} \frac{\sqrt{2}|G(t)|}{|E_\nu(at)|(1+|t|)K_2(t,t)^{1/2}} \\ &\leq    
 \bigg(\sum_{t\in\Tau(\alpha)} \frac{2|G(t)|^2}{K_2(t,t)} \bigg)^{1/2}
\bigg(\sum_{t\in\Tau(\alpha)} \frac{1}{|E_\nu(at)|^2(1+|t|)^2)} \bigg)^{1/2} \\ 
&\<_{C,a,p} \|G\|_{E^2} \\
& \<_{C, a,p} \|F\|_{E,p}\,, 
\end{align*}
where the first inequality is due to the Cauchy-Schwarz inequality, the second inequality due to (\ref{norm-basis}) and the last one due to (\ref{equa-4}). 
\end{proof}


\subsection{Proof of Theorem \ref{gen-int-form-der} - The general case}
\begin{proof}[Step 1] The case $p\in[1,2]$.

\smallskip

As in the proof for the case $p=2$ we can assume that $\alpha=-\pi/2$, which implies $B_\alpha(z)=A(z)$.
Since the space $\H(E^2)$ is closed by differentiation we can apply Proposition \ref{diff-lemma} to conclude that $\p'(x)$ is bounded. By Lemma \ref{concatenation-lemma} we have $\H^p(E^2)\subset \H(E^2)$. Hence, formula (\ref{gen-int-form}) holds for every $F\in\H^p(E^2)$, where the convergence is taken in $\H(E^2)$.
\smallskip

\noindent {\it Step 2.} Preparation for the case $p\in(2,\infty)$.

\smallskip

For $2<p< \infty$ note that if $F\in\H^p(E^2)$
 then $G_w(z)=[F(z)A(w)^2-A(z)^2F(w)]/(z-w)$ belongs to $\H(E^2)$ for every $w\in\C$. By Step 1, we can apply formula (\ref{gen-int-form}) to obtain
 \begin{align}\label{equa-2}
 \begin{split}
 \frac{F(z)}{A(z)^2}-\frac{F(w)}{A(w)^2} =& \sum_n\bigg\{\frac{F(s_n)}{A'(s_n)^2(z-s_n)^2}+ \frac{F'(s_n)}{A'(s_n)^2(z-s_n)}-\frac{F(s_n)A''(s_n)}{A'(s_n)^3(z-s_n)} \\ 
 &\ \ \ \ \ \ \ \ - \frac{F(s_n)}{A'(s_n)^2(w-s_n)^2}- 
\frac{F'(s_n)}{A'(s_n)^2(w-s_n)}+\frac{F(s_n)A''(s_n)}{A'(s_n)^3(w-s_n)}  \bigg\}
\end{split}
 \end{align}
for every $w,z\in\C$. Note that, for every $w\in\C\setminus\{s_n\}$ the sum converges uniformly on compact sets of $\C\setminus\{s_n\}$. If we denote by $K_2(w,z)$ the reproducing kernel of $\H(E^2)$ and use the hypothesis that $\p'(s_n)=|A'(s_n)/B(s_n)|\geq \delta$ for every $n$, together with (\ref{K2-K}), we obtain
\begin{equation}\label{equa-12}
K_2(s_n,s_n)^{1/2}\<_\delta A'(s_n)^2 \,\,\,\, \mbox{ for every } n.
\end{equation}

Since $\H(E^2)$ is closed by differentiation and $v(E^*/E)<0$, we can apply \cite[Theorem A]{Bar2} to conclude that $E'/E\in  H^\infty(\U)$. Again by \cite[Theorem A]{Bar2} the space $\H^p(E^2)$ is closed by differentiation. Since $\p'(x)$ is bounded, we can apply Lemma \ref{tech-lemma-p>2} together with estimate (\ref{equa-12}), to obtain
 \begin{equation}\label{equa-3}
 \sum_n \frac{|F(s_n)|+|F'(s_n)|}{(1+|s_n|)A'(s_n)^2} <\infty  
 \end{equation}
for every $F\in\H^p(E^2)$. Since $|A'(s_n)/B(s_n)|\geq \delta$ for every $n$ we can apply Proposition \ref{diff-lemma} item (4) to conclude that  $|A''(s_n)/A'(s_n)|\<_{D,\delta} 1$ for all $n$. These facts imply that the series
 $$
  \sum_n\frac{F(s_n)}{A'(s_n)^2(z-s_n)^2}+ \frac{F'(s_n)}{A'(s_n)^2(z-s_n)}-\frac{F(s_n)A''(s_n)}{A'(s_n)^3(z-s_n)}
 $$
converges uniformly on compact sets contained in $\C\setminus\{s_n\}$.

By (\ref{equa-2}) we deduce that 
 \begin{equation}\label{equa-5}
 F(z)=c(F)A(z)^2+A(z)^2 \sum_n\frac{F(s_n)}{A'(s_n)^2(z-s_n)^2}+ \frac{F'(s_n)}{A'(s_n)^2(z-s_n)}-\frac{F(s_n)A''(s_n)}{A'(s_n)^3(z-s_n)}
 \end{equation}
 for some complex number $c(F)$. 
 
\smallskip
 
\noindent {\it Step 3.} Finishing the proof for $p\in(2,\infty)$.
 
\smallskip
 
To finish the proof we will show that $c(F)$ is a continuous linear functional over $\H^p(E^2)$ that vanishes in a dense set of functions, hence it is identically zero. By (\ref{swartz-ineq}) we have $|F(i)|\< \|F\|_{E^2,p}$ and  by Lemma \ref{tech-lemma-p>2} we have
$$
 \bigg|\sum_n\frac{F(s_n)}{A'(s_n)^2(i-s_n)^2}+ \frac{F'(s_n)}{A'(s_n)^2(i-s_n)}-\frac{F(s_n)A''(s_n)}{A'(s_n)^3(i-s_n)}\bigg| \< \|F\|_{E^2,p}.
 $$
By (\ref{equa-5}) we conclude that $|c(F)|\< \|F\|_{E^2,p}$, hence $c(\cdot)$ is a bounded linear functional over $\H^p(E^2)$. Since $AB\notin \H(E^2)$ we can apply Lemma \ref{tech-lemma-density} to deduce that the set of functions
$$
\{A(z)B(z)/(z-t_n)\} \cup \{A(z)B(z)/(z-s_n)\}
$$
is dense in  $\H^p(E^2)$. Using formulas (\ref{B/A-form-0}) - (\ref{B/A-form-2}) we see that $c(F)=0$ for any of the above functions, hence $c(\cdot)\equiv 0$. This concludes the proof.
\end{proof}

\noindent {\it Remark:} Note that, by the previous proof for the case $p\in(2,\infty)$, the additional assumption $v(E^*/E)<0$ can be replaced by the following assumption
\begin{equation}\label{dropped-cond}
\sum_{t\in\Tau(\alpha)}\frac{|F(t)|+|F'(t)|}{(1+|t|)K_2(t,t)^{1/2}} \< \|F\|_{E^2,p} \,\,\,\, \mbox{ for every } \,\,\, F\in\H^p(E^2).
\end{equation}
In the next section we shall use condition (\ref{dropped-cond}) to obtain the Theorem \ref{gen-int-form-der} in the range $p\in(2,\infty)$ for homogeneous spaces of entire functions.

\section{Applications}\label{app}

\subsection{Homogeneous de Branges Spaces}\label{hom-space}
There is a variety of examples of de Branges spaces \cite[Chapter 3]{B2} for which Theorem \ref{gen-int-form-der} may be applied. A basic example would be the classical Paley-Wiener space $\H(e^{-i\tau z})$  which gives us the previous results obtained by J. Vaaler in \cite[Theorem 9]{V}. Another interesting family arises in the discussion of \cite[Section 5]{HV}. In the terminology of de Branges \cite[Section 50]{B2}, these are examples of homogeneous spaces, and we briefly review their construction below (see also \cite{B}).

\smallskip

Let $\nu > -1$ be a parameter and consider the real entire functions $A_\nu(z)$ and $B_\nu(z)$ given by
\begin{equation}\label{A-id}
A_{\nu}(z) = \sum_{n=0}^{\infty} \frac{(-1)^n \big(\tfrac12 z\big)^{2n}}{n!(\nu +1)(\nu +2)...(\nu+n)} = \Gamma(\nu +1) \left(\tfrac12 z\right)^{-\nu} J_{\nu}(z)
\end{equation}
and
\begin{equation}\label{B-id}
B_{\nu}(z) = \sum_{n=0}^{\infty} \frac{(-1)^n \big(\tfrac12 z\big)^{2n+1}}{n!(\nu +1)(\nu +2)...(\nu+n+1)} = \Gamma(\nu +1) \left(\tfrac12 z\right)^{-\nu+1} J_{\nu+1}(z),
\end{equation}
where $J_\nu(z)$ denotes the classical Bessel function of the first kind given by
\begin{equation*}\label{bessel-func}
J_\nu(z)=\sum_{n\geq 0} \frac{(-1)^n(\hf z)^{2n+\nu}}{n!\,\Gamma(\nu+n+1)}.
\end{equation*}
If we write $z=x+iy$ then, for every $\nu>-1$, we have
\begin{equation}\label{J-dca-real}
J_\nu(z)=\sqrt{\frac{2}{\pi z}}\, \bigg(\!\cos(z-\nu\pi/2-\pi/4)+e^{|y|}O(1/|z|)\bigg)
\end{equation}
for $x>0$. This estimate can be found in \cite[Section 7.21]{Wat}.

\smallskip

If we write
\begin{equation*}
E_\nu(z) = A_\nu(z) - iB_\nu(z),
\end{equation*}
then the function $E_\nu(z)$ is a Hermite-Biehler function with no real zeros. Moreover, it is of bounded type in $\U$ and of exponential type in $\C$, with $v(E_\nu)=\tau(E_\nu) = 1$. Observe that when $\nu = -1/2$ we have simply $A_{-1/2}(z) = \cos z$ and $B_{-1/2}(z) = \sin z$. 

\smallskip

These special functions also satisfy the following differential equations
\begin{align}\label{nu-diff-eqn}
\begin{split}
A'_\nu(z)&=-B_\nu(z) \\ B'_\nu(z)&=A_\nu(z)-(2\nu+1)B_\nu(z)/z.
\end{split}
\end{align}
By (\ref{A-id}), (\ref{B-id}) and (\ref{J-dca-real}) we have
\begin{equation}\label{homog_eq1}
 |E_{\nu}(x)|^{-2} \simeq_\nu |x|^{2\nu+1}
\end{equation}
and
\begin{equation}\label{AB-est}
|x|^{2\nu+1}|A_\nu(x)B_\nu(x)|=C_\nu\big(|\sin(2x-\nu\pi)|+O(1/|x|)\big)
\end{equation} \\
for $|x|\geq 1$. We conclude that $A_\nu B_\nu\notin\H(E^2_\nu)$. Also, by (\ref{nu-diff-eqn}) we have 
\begin{equation}\label{E'/E-bound}
i\frac{E_\nu'(z)}{E_\nu(z)}=1 - (2\nu+1)\frac{B_\nu(z)}{zE_\nu(z)}.
\end{equation}
for all real $z\in\U$. Hence $[E'_\nu(z)/E_\nu(z)]\in H^\infty(\U)$.

\smallskip

Denoting by $\p_\nu(z)$ the phase function associated with $E_\nu(z)$  and using the fact that $\p_\nu'(t)=\Rep \, [iE_\nu'(t)/E_\nu(t)]$ for all real $t$, we can use (\ref{E'/E-bound}) to obtain
$$
\p'_\nu(t)=1-\frac{(2\nu+1) A_\nu(t)B_\nu(t)}{t|E_\nu(t)|^2}.
$$
Hence,
\begin{equation}\label{phi'-bound}
 \p'_\nu(t) \simeq_\nu 1 \,\,\,\, \mbox{for all real } t.
\end{equation}
For each $F\in\H(E_\nu)$ we have the remarkable identity 
\begin{align}\label{Intro_homog_eq2}
\int_{-\infty}^\infty |F(x)|^{2}\,|E_{\nu}(x)|^{-2}\, \dx = c_\nu \int_{-\infty}^\infty |F(x)|^2 \,|x|^{2\nu+1} \,\dx\,,
\end{align}
with $c_\nu = \pi \,2^{-2\nu-1}\, \Gamma(\nu+1)^{-2}$. Using the fact that $E_\nu(z)$ is of bounded type, we can apply Krein's Theorem (see \cite{Kr} and \cite[Lemma 12]{HV}) together with \eqref{homog_eq1} and \eqref{Intro_homog_eq2} to conclude that $F \in \H(E_\nu)$ if and only if $F$ has exponential type at most $1$ and either side of \eqref{Intro_homog_eq2} is finite. Again, by Krein's Theorem, $F\in\H^p(E^2_\nu)$ if and only if $F(z)$ has exponential type at most $2$ and $F/E_{\nu}^2\in L^p(\R,\dx)$.

\smallskip

For $\nu>-1/2$, the {\it Hankel's integral} for $J_\nu(z)$ is given by
\begin{equation*}\label{Me-So-form}
J_\nu(z) = \frac{({z/2)}^{\nu}}{\Gamma(\nu+1/2)\sqrt{\pi}}\int_{-1}^{1}e^{is z}(1-s^2)^{\nu-\tfrac{1}{2}}\, \ds.
\end{equation*}
This formula can be found in \cite[Section 93]{Bow}.
Using (\ref{A-id}) - \eqref{B-id} and an integration by parts, we deduce the following integral representation
\begin{equation*}\label{E-nu-form}
E_\nu(z)=\frac{\Gamma(\nu+1)}{\Gamma(\nu+1/2)\sqrt{\pi}}\int_{-1}^{1}e^{is z}(1-s^2)^{\nu-\tfrac{1}{2}}(1-s)\,\ds.
\end{equation*}
By simple estimates, we deduce from the above representation that $v(E^*_\nu)=1$ for $\nu>-1/2$. Thus, we cannot directly apply Theorem \ref{gen-int-form-der} for homogeneous spaces in the case $p>2$. Nevertheless, we will prove Theorem \ref{gen-int-form-der} for these homogeneous spaces by verifying that the alternative condition \eqref{dropped-cond} holds.

\begin{lemma}\label{hom-spaces-lemma}
Let $\nu > -1$. The space $\H^p(E_\nu^2)$ satisfies the following properties:
\begin{enumerate}
\item $\H^p(E_\nu^2)\subset \H^q(E_\nu^2)$ if $0<p<q\leq \infty$.
\item $\H^p(E_\nu^2)$ is closed by differentiation for every $p\in[1,\infty]$.
\item If $p\in[1,\infty)$ there exists a constant $C_{\nu,p}>0$ such that
\begin{equation*}
\sum_{t\in\Tau_\nu(\alpha)}\frac{|F(t)|+|F'(t)|}{(1+|t|)K_{2,\nu}(t,t)^{1/2}} \leq C_{\nu,p}\|F\|_{E_\nu^2,p} \,\,\,\, \mbox{ for every } \,\,\, F\in\H^p(E^2_\nu),
\end{equation*}
where the function $K_{2,\nu}(w,z)$ denotes the reproducing kernel of $\H(E_\nu^2)$ and $\Tau_\nu(\alpha)=\{t\in\R:\varphi_\nu(t)\equiv \alpha \, (\!\!\!\!\mod \pi)\}$.
\end{enumerate}
\end{lemma}

\begin{proof} First we prove item (1). Define an auxiliary function $\Psi(z)$ in the following way. If $2\nu+1 < 1$ write $\Psi(z)=E_\sigma(z)^2$ where 
$2\nu+1=-(2\sigma+1)$. If $2\nu+1 \geq 1$, let $k\geq 1$ be a positive integer such that $k\leq 2\nu+1<k+1$ and define 
$\Psi(z)=E_{-3/4}(z)^{4k}E_{\sigma}(z)^2$ where $2\sigma+1=(k-2\nu-1)$. We conclude that $\Psi(z)$ is of exponential type and, by (\ref{homog_eq1}), $|\Psi(x)|\simeq_\nu |x|^{2\nu+1}$ for $|x|\geq 1$. By (\ref{E'/E-bound}) and some simple calculations we have $|\Psi'(t)|\<|\Psi(t)|$ for all real $t$. Also, by redefining $\tilde \Psi(z)=\Psi(az)$ for some $a>0$, we can suppose that $\Psi(z)$ has exponential type 1.

\smallskip

We conclude that $F\in\H^p(E_\nu^2)$ if and only if $F(z)$ is of exponential type at most $2$ and $F\Psi\in L^p(\R,\dx)$. Thus, $\Psi(z)\H^p(E_\nu^2)\subset \PW(3,p)$, where $\PW(3,p)$ is the Paley-Wiener space defined in Subsection \ref{deBranges}. The Plancherel-P\'{o}lya Theorem (see \cite{PP}) implies that $\PW(a,p)\subset \PW(a,q)$ for every $a>0$ and $0<p<q\leq \infty$. We conclude that $F\Psi\in \PW(3,q)$ for every $F\in\H^p(E_\nu)$. This proves item (1).

\smallskip

Now we prove item (2). If $F\in\H^p(E_\nu^2)$ does not have zeros then, since it is of exponential type at most $2$, we deduce that $F(z)=ae^{bz}$ for some $a,b\in\C$ with $|b|\leq 2$. Then $F'=bF$ and trivially $F'\in\H^p(E_\nu^2)$. If $F(z)$ has a zero $z=w$ then $G(z)=F(z)/(z-w)$ is of exponential type at most $2$ and  $G\in L^p(\R,\dx)$. By the Plancherel-P\'{o}lya Theorem, $G'\in L^p(\R,\dx)$ and has exponential type at most $2$. Hence $F'(z)$ has exponential type at most $2$. On the other hand, $F\Psi\in L^p(\R,\dx)$ and again this implies that $(F\Psi)'\in L^p(\R,\dx)$. Since
$F'\Psi=(F\Psi)'-F\Psi'$ and $|\Psi'(t)|\<|\Psi(t)|$ for all real $t$, we conclude that $F'\Psi \in L^p(\R,\dx)$. Hence $F'\in\H^p(E_\nu^2)$.

\smallskip

Finally we prove item (3). By item (2) it is sufficient to prove that
$$
\sum_{t\in\Tau_\nu(\alpha)}\frac{|F(t)|}{(1+|t|)K_{2,\nu}(t,t)^{1/2}} \<_{p,\nu} \|F\|_{E_\nu^2,p}, \,\,\,\, \mbox{ for every } \,\,\, F\in\H^p(E^2_\nu).
$$
By (\ref{phi'-bound}) we  conclude that $K_{2,\nu}(t,t)^{1/2}\simeq |E_\nu(t)|^2$ for all real $t$ and $\Tau_\nu(\alpha)$ is separated with width of separation depending only on $\nu$. We can use H\"older's inequality to conclude that
$$
\sum_{t\in\Tau_\nu(\alpha)}\frac{|F(t)|}{(1+|t|)K_{2,\nu}(t,t)^{1/2}}\<_{p,\nu} \bigg(\sum_{t\in\Tau_\nu(\alpha)}\bigg|\frac{F(t)}{E_\nu(t)^{2}}\bigg|^p\bigg)^{1/p}.
$$
Hence, we only need to show that
\begin{equation}\label{pp-ineq}
\sum_{t\in\Tau_\nu(\alpha)}\bigg|\frac{F(t)}{E_\nu(t)^{2}}\bigg|^p \<_{p,\nu} \int_\R \bigg|\frac{F(t)}{E_\nu(t)^2}\bigg|^p\dt
\end{equation}
for all $F\in\H^p(E^2_\nu)$. Since $\Psi F\in L^p(\R,\dx)$ and $\Tau_\nu(\alpha)$ is separated, we can apply the Plancherel-P\'olya Theorem to obtain 
$$
\sum_{t\in\Tau_\nu(\alpha)}|F(t)\Psi(t)|^p \<_{p,\nu} \int_\R |F(t)\Psi(t)|^p\,\dt
$$
for every $F\in\H^p(E^2_\nu)$. This implies (\ref{pp-ineq}) and concludes the lemma.
\end{proof}

\noindent{\it Remark:} The proof of item (2) is inspired in the proof of \cite[Theorem 20]{CL3}.

\smallskip

From Lemma \ref{hom-spaces-lemma} and condition \eqref{dropped-cond} we conclude the validity of the interpolation formula  (\ref{gen-int-form}) for these homogeneous spaces of entire functions, summarized in the next theorem (with $E(z)=E_\nu(z)$ for $\alpha=0$ and $\alpha=-\pi/2$). Due to identities (\ref{A-id}) - (\ref{B-id}), this can also be seen as an independent contribution to the theory of Bessel functions.

\begin{theorem}\label{int-form-hom}
Let $p\in(0,\infty)$  and $\nu>-1$. Let $F(z)$ be an entire function of exponential type at most $2$ such that
$$
\int_{|t|\geq 1}\big|F(t)|t|^{2\nu+1}\big|^p\,\dt <\infty.
$$
Then
\begin{equation*}\label{int-form-gen-Anu}
\frac{F(z)}{A_\nu(z)^2}=\sum_{A_\nu(s)=0}\bigg\{\frac{F(s)}{A_\nu'(s)^2(z-s)^2}+\frac{F'(s)}{A_\nu'(s)^2(z-s)}\bigg\}+(2\nu+1)\sum_{A_\nu(s)=0} \frac{F(s)}{sA'_\nu(s)^2(z-s)}
\end{equation*}
and
\begin{equation*}\label{int-form-gen-Bnu}
\frac{F(z)}{B_\nu(z)^2}=\sum_{B_\nu(t)=0}\bigg\{\frac{F(t)}{B_\nu'(t)^2(z-t)^2}+\frac{F'(t)}{B_\nu'(t)^2(z-t)}\bigg\} + (2\nu+1)\sum_{\stackrel{B_\nu(t)=0}{t\neq 0}} \frac{F(t)}{tB'_\nu(t)^2(z-t)}\,,
\end{equation*}
where these series converge uniformly on compact sets of $\C$ away from their respective singularities. 
\end{theorem}

\subsection{Extremal Functions}\label{extfunc}
The purpose of this subsection is to prove a uniqueness result for some extremal problems described below. Let $d$ denote the dimension. A set $K\subset\R^d$ is called a convex body if it is compact, convex, symmetric around the origin and has the origin as an interior point. Let $|\cdot|$ denote the Euclidean norm in $\R^d$ and $\B$ the compact Euclidean unit ball. Given a non-negative Borel measure $\mu$ on $\R^d$ and a real-valued function $g(x)$ we denote by $P^+(g,K,\mu)$ the set of measurable real-valued functions $M(x)$ defined on $\R^d$ satisfying the following conditions:

\begin{enumerate}
\item $M(x)$ defines a tempered distribution such that its distributional Fourier transform $\ft M$ is supported on $K$.
\item $g(x)\leq M(x)$ for all $x\in\R^d$.
\item $M-g\in L^1(\R^d,\mu)$.
\end{enumerate}
In this case, we say that $M(x)$ is a band-limited majorant of $g(x)$. In an analogous way we define $P^-(g,K,\mu)$ as the set of minorants.
We are asked to minimize the quantities
\begin{equation}\label{mu-min}
\int_{\R^d} \big\{M(x)-g(x)\big\}\,\d\mu(x)\,\,\,\, \mbox{ and }\,\,\,\, \int_{\R^d} \big\{g(x)-L(x)\big\}\,\d\mu(x)
\end{equation}
among all functions $M\in P^+(g,K,\mu)$ and $L\in P^-(g,K,\mu)$. And, if the minimum is attained, characterize the set of extremal functions. We call $M(x)$ (or $L(x)$) an {\it extremal function} if it minimizes the quantity (\ref{mu-min}).\smallskip

The problem becomes treatable if we consider radial functions. For instance, we consider the situation where $K=\B$, the function $g(x)$ is radial, and
\begin{equation}\label{mu-E}
\d\mu_E(x)=2\bigg(|E(|x|)|^2|x|^{d-1}\big|S^{d-1}\big|\bigg)^{-1}\dx\,,
\end{equation}
where  $\big|S^{d-1}\big|$ denotes the area of the $(d-1)$-dimensional sphere. Also, in this subsection, $E(z)$ will always denote a Hermite-Biehler function of bounded type and mean type equal to $\pi$ such that $\H(E^2)$ is closed by differentiation and $\p'(t)$ is bounded away from zero over the zero set of $A(z)$ and $B(z)$. We also assume that $E^*(-z)=E(z)$ and $AB\notin\H(E^2)$. This implies that the companion functions $A(z)$ and $B(z)$ are respectively even and odd and $A,B\notin \H(E)$. By Krein's Theorem, $E(z)$ is of exponential type with $\tau(E)=v(E)=\pi$, and $F\in\H(E)$ if and only if $F(z)$ is of exponential type at most $\pi$ and $F/E\in L^2(\R,\dx)$ (see \cite[Lemmas 9 and 12]{HV}).

\smallskip

These restrictions reduce the multidimensional problem to a one-dimensional problem and allow us to use de Branges space techniques. Constructions of extremal band-limited approximations of radial functions in several variables were studied in \cite{CL3,CL4,HV}. In particular, E. Carneiro and F. Littmann \cite{CL3, CL4} were able to explicitly construct a pair of radial functions $M\in P^+(g,\B,\mu_E)$ and $L\in P^-(g,\B,\mu_E)$ that minimize the quantities in
(\ref{mu-min}), where $\mu_E$ is given by (\ref{mu-E}), $E(z)=E_\nu(z)$ and $g(x)$ belongs to a vast class of radial functions with exponential or Gaussian subordination.

\smallskip

For the sake of completeness we state here a classical theorem about tempered distributions with Fourier transform supported on a ball. This result can be found in \cite[Theorem 7.3.1]{Hor}.

\begin{theorem}[Paley--Wiener--Schwartz]\label{thm-PWS} 
Let $F$ be a tempered distribution such that the support of $\ft{F}$ is contained in $\B$. Then $F:\C^d\to\C$ is an entire function and there exist $N,C>0$ such that
$$
|F(x+iy)| \leq C(1+|x+iy|)^Ne^{2\pi|y|}
$$
for every  $x+iy\in\C^d$. 

Conversely, every entire function $F:\C^d\to\C$ satisfying an estimate of this form defines a tempered distribution with Fourier transform supported on $\B$.
 
\end{theorem}

The next propositions give an interpolation condition for a band-limited majorant or minorant to be extremal and unique in radial case. We highlight the fact that the uniqueness part below is a novelty in this multidimensional theory, and makes a crucial use of our interpolation formulas. This enhances the extremal results obtained in \cite{CL3,CL4}.

\begin{proposition}\label{min-uniq}
Let $g(x)=g(|x|)$ be a radial function that is differentiable for $x\neq 0$. Suppose that $P^+(g,\B,\mu_E)\neq \emptyset$ and there exists a radial function $L\in P^-(g,\B,\mu_E)$ such that $L(x)=g(x)$ whenever $A(|x|)=0$. Then $L$ is extremal and unique among the set of entire functions  on $\C^d$ whose restriction to $\R^d$ is radial.
 
\end{proposition}

\begin{proposition}\label{maj-uniq}
Let $g(x)=g(|x|)$ be a radial function that is differentiable for $x\neq 0$. Suppose that $P^-(g,\B,\mu_E)\neq \emptyset$ and there exists a radial function $M\in P^+(g,\B,\mu_E)$ such that $M(x)=g(x)$ whenever $B(|x|)=0$. Then $M$ is extremal and unique among the set of entire functions  on $\C^d$ whose restriction to $\R^d$ is radial.

\end{proposition}

We only prove Proposition \ref{maj-uniq} since the other is analogous.
\begin{proof}
{\it Optimality.}
\smallskip

Fix $L\in P^-(g,\B,\d\mu_E)$. Let $SO(d)$ denote the compact topological group of real orthogonal $d \times d$ matrices with determinant $1$, with associated probability Haar measure $\sigma$. If $R\in P^+(g,\B,\mu_E)$, then
$$
\tilde R(x)=\int_{SO(d)}R(\rho x)\,\d\sigma(\rho)
$$
is radial, belongs to $P^+(g,\B,\mu_E)$ and
\begin{equation}\label{maj-id-1}
\int_{\R^d} \big\{\tilde R(x)-M(x)\big\}\, \d\mu_E(x)=\int_{\R^d} \big\{R(x)-M(x)\big\}\, \d\mu_E(x).
\end{equation}
In the same way, we define $\tilde L(x)$ as the radial symmetrization of $L(x)$. Again we have $\tilde L\in P^-(g,\B,\d\mu_E)$.
Define $m(t)=M(te_1)$,  $l(t)=\tilde L(te_1)$ and $r(t)=\tilde R(te_1)$ for all real $t$, where $e_1=(1,0,...,0)$. We can apply the Paley-Wiener-Schwartz Theorem to conclude that these functions extend to $\C$ as entire functions of exponential type at most $2\pi$. By (\ref{mu-E}) we obtain that
\begin{equation}\label{maj-id-2}
\int_{\R^d} \big\{\tilde R(x)-M(x)\big\}\d\mu_E(x) = \int_{\R} \{r(t)-m(t)\}/|E(t)|^2\,\dt.
\end{equation}
We claim that $r-m=pp^*-qq^*$ for $p,q\in\H(E)$. Since $m(x)-l(x)\geq 0$ and $r(x)-l(x)\geq 0$ for all real $x$, we conclude that there exists two entire functions $p(z)$ and $q(z)$ of exponential type at most $\pi$ such that $m(z)-l(z)=p(z)p^*(z)$ and $r(z)-l(z)=q(z)q^*(z)$ (see \cite[Theorem 13]{B2}). Since $m-l$ and $r-l$ belong to $L^1(\R,|E(x)|^{-2}\dx)$ we conclude that $p,q\in\H(E)$. We can apply formula (\ref{norm-basis}) to obtain that

\begin{align}\label{maj-id-3}
\begin{split}
\int_{\R} \{r(t)-m(t)\}|E(t)|^{-2}\dt & = \int_{\R} \frac{|p(t)|^2-|q(t)|^2}{|E(t)|^2}\,\dt =\sum_{B(t)=0}\frac{|p(t)|^2-|q(t)|^2}{K(t,t)}\\ & = \sum_{B(t)=0}\frac{r(t)-m(t)}{K(t,t)} =\sum_{B(t)=0}\frac{r(t)-g(|t|)}{K(t,t)}\\
&  \geq 0\,,
\end{split}
\end{align}
where the last equality is due to the interpolation condition of $M(x)$, that is, $M(x)=g(x)$ whenever $B(|x|)=0$. By (\ref{maj-id-1}), (\ref{maj-id-2}) and (\ref{maj-id-3}) we conclude that $M(x)$ is extremal.
\smallskip

\noindent {\it Uniqueness.}
\smallskip

Inequality (\ref{maj-id-3}) implies that if $R\in P^+(g,\B,\mu_E)$ is radial and extremal, then $r(t)=g(|t|)$ whenever $B(t)=0$. Since $x\in\R^d\mapsto g(x)=g(|x|)$ is radial and differentiable for $x\neq 0$ we conclude that $r'(t)=\sgn(t)g'(|t|)$ if $B(t)=0$ and $t\neq 0$. Also $r'(0)=0$. Since $f:=(m-r)\in\H^1(E^2)$ and $f(t)=f'(t)=0$ whenever $B(t)=0$, by Theorem \ref{gen-int-form-der} we conclude that $f\equiv 0$. Hence, $M(x)$ is unique.
\end{proof}

\noindent {\it Remark:} In some cases $g(x)$ may have a singularity at $x=0$, for instance if $\lim _{x\to 0} g(x)=\infty$. Thus, only the minorant problem is well-posed, that is $P^+(g,\B,\mu_E)=\emptyset$. However, in the case of homogeneous spaces the previous proposition will still hold. In \cite[Corollary 23]{CL3}, E. Carneiro and F. Littmann proved that every $f\in\H^1(E^2_\nu)$, not necessarily non-negative on the real axis, can be represented as $f=pp^*-qq^*$ for $p,q\in\H(E_\nu)$. We can easily see that this representation is sufficient to prove the previous propositions for $E(z)=E_\nu(z)$ in the case when $g(x)$ has a singularity.

\medskip

\section*{Acknowledgements}
I am deeply grateful to my advisor Emanuel Carneiro for encouraging me to work on this problem and for all the fruitful discussions on the elaboration of this paper.

The author also acknowledges the support from CNPq--Brazil and FAPERJ--Brazil.


\medskip

\begin{thebibliography}{99}  

\bibitem{ABR}
S. Axler, P. Bourdon and W. Ramey,
\newblock {\it Harmonic Function Theory},
\newblock Graduated Texts in Mathematics 137, 1992.

\bibitem{Bar}
A. Baranov,
\newblock Differentiation in De Branges Spaces and Embedding Theorems, 
\newblock Journal of Mathematical Sciences 101, No. 2 (2000), 2881--2913.

\bibitem{Bar2}
A. Baranov,
\newblock Estimates of the $L^p$-Norms of Derivatives in Spaces of Entire Functions, 
\newblock Journal of Mathematical Sciences 129, No. 4 (2005), 3927--2943.

\bibitem{Bow}
F. Bowman,
\newblock {\it Introduction to Bessel Functions},
\newblock Dover Publications, 1958.

\bibitem{B}
L. de Branges,
\newblock Homogeneous and Periodic Spaces of Entire Functions,
\newblock Duke Math. Journal 29 (1962), 203--224.

\bibitem{B2} 
L. de Branges,
\newblock {\it Hilbert Spaces of Entire Functions},
\newblock Prentice-Hall Series in Modern Analysis, 1968.

\bibitem{CC}
E. Carneiro and V. Chandee,
\newblock Bounding $\zeta(s)$ in the Critical Strip,
\newblock J. Number Theory 131 (2011), 363--384.

\bibitem{CCLM}
E. Carneiro, V. Chandee, F. Littmann and M. B. Milinovich,
\newblock Hilbert Spaces and the Pair Correlation of Zeros of the Riemann Zeta-Function,
\newblock J. Reine Angew. Math (to appear).

\bibitem{CCM}
E. Carneiro, V. Chandee and M. B. Milinovich,
\newblock Bounding $S(t)$ and $S_1(t)$ on the Riemann Hypothesis,
\newblock Math. Ann. 356 (2013), 939--968.

\bibitem{CCM2}
E. Carneiro, V. Chandee and M. Milinovich,
\newblock A note on the zeros of zeta and $L$-functions,
\newblock Preprint.

\bibitem{CG}
E. Carneiro and F. Gon\c{c}alves,
Extremal Problemas in de Branges Spaces: The Case of Truncated and Odd Functions,
\newblock Preprint.

\bibitem{CL}
E. Carneiro and F. Littmann,
\newblock Bandlimited approximations to the truncated Gaussian and applications,
\newblock Constr. Approx. 38 (2013), 19--57. 

\bibitem{CL3}
E. Carneiro and F. Littmann,
\newblock Extremal Functions in de Branges and Euclidean Spaces,
\newblock Adv. Math. 260 (2014), 281--349.

\bibitem{CL4}
E.~Carneiro and F.~Littmann,
\newblock Extremal functions in de Branges and Euclidean Spaces II,
\newblock Preprint.

\bibitem{CLV}
E. Carneiro, F. Littmann, and J. D. Vaaler,
\newblock Gaussian Subordination for the Beurling-Selberg Extremal Problem,
\newblock Trans.\ Amer.\ Math.\ Soc. 365 (2013), 3493--3534. 

\bibitem{CV2}
E.~Carneiro and J.~D.~Vaaler,
\newblock Some Extremal Functions in Fourier Analysis II,
\newblock Trans. Amer. Math. Soc. 362 (2010), 5803--5843.

\bibitem{CV3}
E.~Carneiro and J.~D.~Vaaler,
\newblock Some Extremal Functions in Fourier Analysis III,
\newblock Constr. Approx. 31, No. 2 (2010), 259--288.

\bibitem{CS}
V. Chandee and K. Soundararajan, 
\newblock Bounding $|\zeta(1/2+it)|$ on the Riemann Hypothesis, 
\newblock Bull. London Math. Soc. 43, No. 2 (2011), 243--250.

\bibitem{Co}
W. S. Cohn, 
\newblock Radial Limits and Star--Invariant Subspaces of Bounded Mean Oscillation,
\newblock Amer. J. Math. 108 (1986), 719--749.


\bibitem{Ga}
P. X. Gallagher,
\newblock Pair Correlation of Zeros of the Zeta Function,
\newblock J. Reine Angew. Math. 362 (1985), 72--86.

\bibitem{GG}
D. A. Goldston and S. M. Gonek,
\newblock A Note on S(t) and The Zeros of the Riemann Zeta-function,
Bull. London Math. Soc. 39 (2007), 482--486.

\bibitem{GMK}
F. Gon\c{c}alves, M. Kelly and J. Madrid,
\newblock One-Sided Band-Limited Approximations in Euclidean Spaces of Some Radial Functions,
\newblock Preprint.

\bibitem{GV}
S. W. Graham and J.~D.~Vaaler,
\newblock A Class of Extremal Functions for the Fourier Transform,
\newblock Transactions of the American Mathematical Society 265, No. 1 (1985), 283--302.


\bibitem{HV}
J. Holt and J. D. Vaaler, 
\newblock The Beurling-Selberg Extremal Functions for a Ball in the Euclidean space, 
\newblock Duke Mathematical Journal 83 (1996), 203--247.

\bibitem{Hor}
L. H\"{o}rmander,
\newblock {\it The Analysis of Linear Partial Differential Operators I},
\newblock Springer-Verlag, 1983.

\bibitem{Kr}
M. G. Krein,
\newblock A Contribution to the Theory of Entire Functions of Exponential Type,
\newblock Bull. Acad. Sci. URSS S\'er. Math. [Izvestiya Akad. Nauk. SSSR] 11 (1947), 309--326.

\bibitem{L1} 
F.~Littmann,
\newblock Entire Approximations to the Truncated Powers,
\newblock Constr. Approx. 22, No. 2 (2005), 273--295.

\bibitem{L3}
F.~Littmann,
\newblock Entire majorants via Euler-Maclaurin summation,
\newblock Trans. Amer. Math. Soc. 358, No. 7 (2006), 2821--2836.

\bibitem{L2} 
F. Littmann,
\newblock Quadrature and Extremal Bandlimited Functions,
\newblock SIAM J. Math. Anal. 45, No. 2 (2013), 732--747.

\bibitem{LSp}
F. Littmann and M. Spanier,
\newblock Extremal functions with vanishing condition,
\newblock preprint at http://arxiv.org/abs/1311.1157.

\bibitem{LS}
Y. Lyubarskii and K. Seip, 
\newblock Weighted Paley-Wiener Spaces,
\newblock J. Amer. Math. Soc. 15, No. 4 (2002), 979--1006.

\bibitem{M2}
H.~L.~Montgomery,
\newblock {\it Ten Lectures on the Interface Between Analytic Number Theory and Harmonic Analysis},
\newblock CBMS No. 84, Amer. Math. Soc., Providence, 1994.

\bibitem{MV} 
H.~L.~Montgomery and R.~C.~Vaughan,
\newblock Hilbert's Inequality,
\newblock J. London Math. Soc. 8, No. 2 (1974), 73--81.

\bibitem{OS}
J. Ortega-Cerd\`a and K. Seip,
\newblock Fourier frames,
\newblock Ann. of Math. (2) 155, No. 3 (2002), 789--806.

\bibitem{PP}
M.~Plancherel and G.~Polya,
\newblock Fonctions Enti\'eres et Int\'egrales de Fourier Multiples (Seconde partie),
\newblock Comment. Math. Helv. 10 (1938), 110--163.

\bibitem{S2}
A.~Selberg,
\newblock Lectures on Sieves, {\em Atle Selberg: Collected Papers}, Vol. II,
\newblock Springer-Verlag, 1991.

\bibitem{V}
J.~D.~Vaaler,
\newblock Some Extremal Functions in Fourier analysis,
\newblock Bull. Amer. Math. Soc. 12 (1985), 183--215.

\bibitem{Wat}
G. N. Watson,
\newblock {\it A Treatise on the Theory of Bessel Functions},
\newblock Cambridge Mathematical Library Edition, 1995.


\end{thebibliography}
\end{document}